\documentclass[11pt,a4paper]{article}
\usepackage[T1]{fontenc}
\IfFileExists{lmodern.sty}{\usepackage{lmodern}}{}
\usepackage[margin=27mm,headheight=15pt]{geometry}
\usepackage{amsmath,amssymb,amsthm}
\usepackage{fancyhdr,needspace}
\usepackage[pdfborder={0 0 0},hyperfootnotes=false]{hyperref}
\usepackage{tikz-cd}
\hypersetup{pdftitle={Composition factors and finite representation type},pdfauthor={Liang Chen, Zhenxian Chen, Yongyi Wen, Weiheng Xia},pdfkeywords={Artin algebras, composition factors, finite representation type, representation embeddings, graded modules}}
\newtheorem{theorem}{Theorem}[section]
\newtheorem{lemma}[theorem]{Lemma}
\newtheorem{proposition}[theorem]{Proposition}
\theoremstyle{remark}

\theoremstyle{definition}
\newtheorem*{definition}{Definition}
\newtheorem*{conjecture}{Conjecture}
\theoremstyle{plain}
\newcommand{\X}{\mathfrak X}
\DeclareMathOperator{\rad}{rad}
\DeclareMathOperator{\soc}{soc}
\DeclareMathOperator{\Hom}{Hom}
\DeclareMathOperator{\End}{End}
\DeclareMathOperator{\Ext}{Ext}

\DeclareMathOperator{\ann}{ann}
\DeclareMathOperator{\im}{Im}
\DeclareMathOperator{\topm}{top}
\DeclareMathOperator{\length}{\ell}
\DeclareMathOperator{\coker}{Coker}
\newcommand{\D}{\operatorname{D}}
\newcommand{\modu}{\text{-}\mathrm{mod}}

\title{\bfseries Composition factors and finite representation type}
\author{Liang Chen \quad Zhenxian Chen \quad Yongyi Wen$^{*}$ \quad Weiheng Xia}
\date{}
\begin{document}
\maketitle
\begingroup
\renewcommand{\thefootnote}{}
\footnotetext{\textsuperscript{*}\,Corresponding author. Email:
\href{mailto:2240501002@cnu.edu.cn}{2240501002@cnu.edu.cn}.
\par\noindent\hspace*{1.8em}%
\par\noindent\hspace*{1.8em}%
}
\endgroup
\begin{abstract}
We prove that an Artin algebra has finite representation type if its
indecomposable finitely generated modules are determined, up to
isomorphism, by their composition factors.
This gives an affirmative answer to the conjecture on composition factors
and finite representation type recorded by Auslander, Reiten and
Smal{\o}.
\end{abstract}

\tableofcontents

\section{Introduction}

Let $A$ be an Artin algebra, and let $S_1,\ldots,S_s$ be representatives
of the isomorphism classes of simple $A$-modules. For a finitely generated
$A$-module $M$, the vector
\[
 [M]=\bigl([M:S_1],\ldots,[M:S_s]\bigr)
\]
records the multiplicities of the simple modules in a composition series
of $M$. It represents the class of $M$ in the Grothendieck group
$K_0(A)$ with respect to the basis given by the simple modules.
Determining how much information about an indecomposable module is
retained by this vector is a classical problem in the representation
theory of algebras. We say that $A$ satisfies property $\X$, and write
$\X(A)$, if any two indecomposable finitely generated $A$-modules $M$
and $N$ with $[M]=[N]$ are isomorphic. Thus $\X$ requires the
composition-factor vector to distinguish all indecomposable modules,
with no restriction on their homological properties or on the
Auslander--Reiten components in which they lie.

The basic example comes from Gabriel's theorem: for a path algebra of
Dynkin type over an algebraically closed field, the dimension vectors
of indecomposable modules are in bijection with the positive roots of
the corresponding root system \cite{Gabriel1972}. More generally,
every representation-finite hereditary Artin algebra satisfies $\X$;
see \cite[Chapter~VIII, Corollary~2.4]{ARS}. Happel proved the
corresponding uniqueness result for representation-directed
finite-dimensional algebras \cite{Happel1982}. Short chains and short
cycles provide further criteria \cite{AuslanderReiten1985,
ReitenSkowronskiSmalo1993}. In particular, if an indecomposable module
$M$ does not lie on a short cycle, then any indecomposable module with
the same composition factors as $M$ is isomorphic to $M$
\cite[Chapter~IX, Theorem~2.1]{ARS}. Consequently, an Artin algebra
whose module category has no short cycles satisfies $\X$
\cite[Chapter~IX, Corollary~2.3]{ARS}. Here a short cycle consists of
two nonzero nonisomorphisms $M\to N\to M$ between indecomposable
modules; the maps need not be irreducible.

Further classes arise from tilting theory. Over an algebraically closed
field, Happel showed that representation-finite piecewise hereditary
algebras satisfy $\X$ \cite[Corollary~7.7]{Happel1987}. Here piecewise
hereditary means that the bounded derived category is triangle
equivalent to that of a finite-dimensional hereditary algebra; this
class includes iterated tilted algebras of Dynkin type. Ringel proved
that all representation-finite cluster-tilted algebras over an
algebraically closed field satisfy $\X$
\cite[Theorem~1]{Ringel2011}, and Geng and Peng obtained an independent
proof \cite{GengPeng2012}. These results show that composition-factor
uniqueness also occurs in classes whose quivers may contain oriented
cycles.

There is also a preservation result involving Cohen--Macaulay Auslander
algebras. Chen and Lu proved that, if $\Lambda=kQ/I$ is a gentle algebra
over an algebraically closed field, $Q$ has no loops, and $\X(\Lambda)$
holds, then its Cohen--Macaulay Auslander algebra
$\operatorname{Aus}(\operatorname{Gproj}\Lambda)$ satisfies $\X$
\cite{ChenLu2019}. The conclusion concerns all indecomposable modules
over the new algebra. In particular, combining their result with the
cluster-tilted case gives $\X$ for the Cohen--Macaulay Auslander
algebras of cluster-tilted algebras of type $A_n$.

Elementary module classifications supply additional examples.
Every local Nakayama algebra satisfies $\X$, since its indecomposable
modules are determined by their lengths
\cite[Chapter~VI, Theorem~2.1]{ARS}. The same classification shows that
if $C_n$ is the oriented cycle on $n\ge2$ vertices and $J$ is the arrow
ideal of $kC_n$, then
\[
 \X(kC_n/J^L)\quad\Longleftrightarrow\quad L<n
 \qquad(L\ge1).
\]
Indeed, an indecomposable module of length less than $n$ has consecutive,
distinct simple composition factors, which determine its top and
length. If $L\ge n$, the length-$n$ quotients of the projective covers
of different simple modules are nonisomorphic, although each contains
every simple module exactly once. Thus representation-finiteness alone
does not imply $\X$, even for Nakayama algebras. Another immediate
consequence of a classical classification is that a string algebra
whose underlying unoriented graph is a tree satisfies $\X$: there are
no bands, and the support of a string determines its path up to
inversion \cite{ButlerRingel1987}.

Composition factors have also been used to study specified parts of
module categories. Zhang proved uniqueness within each
Auslander--Reiten component of a wild hereditary Artin algebra
\cite{Zhang1990}. Skowro\'nski and Zwara studied the numbers of discrete
indecomposable modules with prescribed composition factors over tame
algebras \cite{SkowronskiZwara1997}. For generalized standard almost
cyclic coherent components without external short paths, Malicki
obtained uniqueness for indecomposable modules without self-extensions
\cite{Malicki2006}. Jaworska, Malicki and Skowro\'nski gave conditions
under which an entire Auslander--Reiten component is determined by the
composition-factor vectors of its modules
\cite{JaworskaMalickiSkowronski2012}. In higher homological algebra,
Reid established composition-factor criteria for indecomposable objects
in suitable $d$-cluster tilting subcategories \cite{Reid2020}.
These statements concern restricted collections of modules or the
identification of components and do not, by themselves, assert $\X$
for the whole algebra.

The examples above lead to a converse finiteness question, recorded in
\cite[p.~409, Conjecture~(4)]{ARS}. We refer to it as the
\emph{composition-factor conjecture}.

\begin{conjecture}
An Artin algebra satisfying $\X$ has finite representation type.
\end{conjecture}

The conjecture is closely related to the Brauer--Thrall theorems.
The first Brauer--Thrall theorem asserts that a representation-infinite
Artin algebra has indecomposable modules of arbitrarily large composition
length; see \cite{Ringel2008BT1}. This unboundedness does not contradict
$\X$, since the property compares modules with the same
composition-factor vector. Over an algebraically closed field, the
second Brauer--Thrall theorem supplies the required multiplicity:
a representation-infinite finite-dimensional algebra has infinitely
many pairwise nonisomorphic indecomposable modules of some fixed
dimension \cite{Bautista1985,Bongartz1985,Smalo1980}. There are only
finitely many composition-factor vectors in that dimension, so such
an algebra cannot satisfy $\X$. Related developments include the
absence of gaps in the lengths of indecomposable modules
\cite{Bongartz2013NoGaps}; see also the survey
\cite{Bongartz2014Survey}.

The fixed-dimension assertion used in this argument does not extend to
arbitrary ground fields. In particular, over a finite field there are
only finitely many module structures of any prescribed finite
dimension, up to isomorphism. Thus the algebraically closed case does
not settle the conjecture for general Artin algebras. Blasco established
the conjecture for commutative, hereditary, and radical-square-zero
Artin algebras \cite{Blasco}. In the commutative case, $\X$ is
equivalent to the ring being an artinian principal ideal ring, and in
the hereditary case it is equivalent to finite representation type.
The radical-square-zero result is an implication from $\X$ to finite
representation type; its converse fails, as the example
$kC_2/J^2$ above already shows.

The aim of this paper is to prove the conjecture for arbitrary Artin
algebras. Our main result is the following.

\begin{theorem}\label{artin:main}
Let $A$ be an Artin algebra. If $\X(A)$ holds, then $A$ has finite
representation type.
\end{theorem}

The proof begins with a structural reduction. After passing to a basic
algebra over a commutative artinian local base ring, we show that
property $\X$ forces the primitive corners to be local Nakayama
algebras and imposes restrictions on their interaction with the
off-diagonal components. Replacing these corners by truncated skew
polynomial algebras produces a finite-dimensional algebra over the
residue field. The construction gives a bijection between module
isomorphism classes that is compatible with direct sums and preserves
composition multiplicities. It therefore preserves indecomposability
and reduces the finiteness problem to algebras over fields.

The finite-field and infinite-field cases require different arguments.
Over a finite field, property $\X$ gives a polynomial upper bound for
the number of indecomposable modules of bounded dimension. After a
finite splitting-field extension, representation embeddings
\cite{Bongartz} produce exponentially many pairwise nonisomorphic
indecomposable modules if the algebra is representation-infinite,
yielding a contradiction. Related representation embeddings over
perfect fields appear in \cite{BautistaPerezSalmeron2025}.
Over an infinite field, polynomial module families and a first-order
splitting argument show that, in the absence of simple self-extensions,
every indecomposable module is a brick and the algebra is
representation-finite. We then apply this result to categories of
graded modules with support in a fixed finite set. A multigrading on
the replacement algebra allows us to bound, uniformly up to shift,
the supports of all indecomposable graded modules in terms of the
Loewy lengths of its primitive corners. Results of Gordon and Green
\cite{GordonGreen1982} and Farnsteiner \cite{Farnsteiner2009} then give
finite representation type. Further connections between gradability
and representation type are developed by Dugas \cite{Dugas}.

The paper is organized as follows. Section~\ref{sec:preliminaries}
collects preliminary results on Artin algebras, commutative rings,
and graded modules. Section~\ref{c2:sec:replacement} establishes
the reduction to finite-dimensional algebras over fields.
Section~\ref{sec:field-cases} treats the finite-field and infinite-field
cases separately. The proof of Theorem~\ref{artin:main} is completed
in Section~\ref{c2:sec:completion}.

\section{Preliminaries}\label{sec:preliminaries}

Section \ref{sec:preliminaries} collects preliminary results on Artin algebras, commutative rings and graded algebras.

\subsection{Artin algebras}\label{sec:prelim-artin}

An \emph{Artin algebra} is an algebra that is finitely generated as a
module over a commutative artinian ring acting centrally. Unless otherwise specified, modules are left modules. We consider finitely generated modules over Artin algebras, except for the auxiliary bimodules and polynomial families used in
Section \ref{sec:field-cases}. $\Hom$ and $\Ext$ involving these auxiliary modules are computed in the category of
all modules. An
algebra is \emph{representation-finite} if it has only finitely many
isomorphism classes of indecomposable finitely generated modules.
We write $J(A)$ for the Jacobson radical of $A$, $\ell(M)$ for the
composition length of $M$, and $A\modu$ for the category of finitely
generated left $A$-modules. A \emph{brick} is a nonzero module whose
endomorphism ring is a division ring.

Morphisms are composed from left to right: for $f:U\to V$ and
$g:V\to W$, their composite is $fg:U\to W$. We write module
homomorphisms on the right, so that $(x)(fg)=((x)f)g$.
Functors act on the left and are composed from right to left.

The first lemma gives a useful criterion for an extension of bricks
to be indecomposable. We then apply it to an indecomposable nonbrick
of minimal length.

\begin{lemma}
\label{c2:lem:brick-extension}
Let $U$ and $V$ be bricks, and suppose that there is a monomorphism
$u:U\longrightarrow V$.  The middle term of any nonsplit exact sequence
\[
  0\longrightarrow V\xrightarrow{i}E\xrightarrow{\pi}U
     \longrightarrow0
\]
is indecomposable and is not a brick.
\end{lemma}

\begin{proof}
For every $a\in \End_{A}(E)$, we claim that $ia\pi=0$.
Otherwise, $ia\pi u\neq 0$, so $ia\pi u$ is an isomorphism since
$V$ is a brick. Hence $i$ is a split monomorphism, a contradiction.
Thus we obtain the following commutative diagram with exact rows:
\[
\begin{tikzcd}[column sep=large, row sep=large]
0 \arrow{r}
  & V \arrow{r}{i} \arrow[dashed]{d}{f}
  & E \arrow{r}{\pi} \arrow{d}{a}
  & U \arrow{r} \arrow[dashed]{d}{g}
  & 0 \\
0 \arrow{r}
  & V \arrow{r}{i}
  & E \arrow{r}{\pi}
  & U \arrow{r}
  & 0
\end{tikzcd}
\]
If $a^2=a$, then $f^2=f$ and $g^2=g$. Since $V$ and $U$ are bricks,
each of $f,g$ is either zero or the identity homomorphism.
The cases $(f, g) = (1, 0)$ and $(0, 1)$ would split the sequence. If $f = g = 0$, then
$a^2 = 0$, so $a = 0$; applying the same argument to $1-a$ treats the case $f = g = 1$.
Thus $a$ is either zero or the identity homomorphism, and $E$ is
indecomposable. Since $\pi ui$ is nonzero and $(\pi ui)^2=0$,
$E$ is not a brick.
\end{proof}

\begin{lemma}\label{c2:lem:shortest-nonbrick}
Suppose that $A$ satisfies $\X$ and that $M$ has minimal composition
length among the indecomposable $A$-modules that are not bricks.
There is a nonsplit exact sequence
\begin{equation}\label{c2:eq:shortest-selfextension}
 \xi:\quad 0\longrightarrow U\longrightarrow M\longrightarrow U
 \longrightarrow0
\end{equation}
with $U$ a brick. Its pullback along the inclusion of any proper
submodule of $U$ is split.
\end{lemma}
\begin{proof}
The local Artin ring $\End_A(M)$ has nonzero radical. Choose a
nonzero element $f$ of that radical for which $\ell(\im f)$ is
minimal, and put
\[
 I=\im f,\qquad K=\ker f,\qquad C=\coker f.
\]
Since $f$ is nilpotent, $f^2\ne0$ would give a nonzero radical
endomorphism with smaller image. Hence $f^2=0$ and $I\subseteq K$.
Clearly, $I$ is indecomposable.
Since $\ell(I)<\ell(M)$, the module $I$ is a brick.

Write $K=\bigoplus_{j=1}^qK_j$ with each $K_j$ indecomposable.
These modules are bricks by the minimality of $M$. Choose a nonzero
component $u:I\to K_j$ of the inclusion $I\hookrightarrow K$.
The composite $M\twoheadrightarrow I\xrightarrow{u}K_j
\hookrightarrow M$ is nonzero and noninvertible. Minimality of
$\ell(I)$ therefore forces $u$ to be injective.
Pushing out $0\to K\to M\to I\to0$ along the projection onto $K_j$
gives
\[
 0\longrightarrow K_j\longrightarrow
 M\big/\bigoplus_{l\ne j}K_l\longrightarrow I\longrightarrow0.
\]
This sequence is nonsplit, since a splitting would extend the
projection $K\to K_j$ to $M\to K_j$ and make $K_j$ a direct summand
of $M$. Lemma~\ref{c2:lem:brick-extension} makes its middle term an
indecomposable nonbrick. Thus $q=1$, by the minimality of $M$, and
$K$ is a brick.

Dually, $C$ is also a brick. Since $K$ and $C$ have the same composition factors,
property $\X$ gives $K\cong C$. The natural map
$K\hookrightarrow M\twoheadrightarrow C$ has kernel $I\ne0$;
as a noninvertible map between isomorphic bricks, it must be zero.
Hence $K=I$. Taking $U=I$ gives
\eqref{c2:eq:shortest-selfextension}, which is nonsplit because $M$
is indecomposable.

Let $V\subsetneq U$, and write $V$ as a direct sum of
indecomposable modules $V_j$. Each nonzero $V_j$ is a brick and
embeds into $U$. If the pullback of $\xi$ to $V_j$ were nonsplit,
Lemma~\ref{c2:lem:brick-extension} would give an indecomposable
nonbrick of length $\ell(U)+\ell(V_j)<\ell(M)$. This is impossible.
Additivity of $\Ext_A^1(-,U)$ on finite direct sums now proves that
the pullback to $V$ is split.
\end{proof}

\subsection{Commutative rings}\label{sec:prelim-commutative}

The results in this subsection will be used to study families of modules over polynomial rings.
No finite-dimensionality assumption is imposed on the commutative base rings; noetherianity and finite generation are assumed where stated. We begin with Krull's intersection
theorem; see \cite[Theorem~8.10]{CRT}.

\begin{lemma}\label{krull}
Let $R$ be a commutative Noetherian local ring, let $I$ be a proper
ideal of $R$, and let $M$ be a finitely generated $R$-module. Then
\[
 \bigcap_{n\geq0}I^nM=0.
\]
\end{lemma}

\begin{lemma}\label{aux}
Let $R$ be a commutative Noetherian integral domain with fraction field
$K$, and let $I\subsetneq R$ be a proper ideal. Let $L,M$ be finitely
generated $R$-modules such that
\[
 \dim_K(L\otimes_RK)=\dim_K(M\otimes_RK).
\]
Suppose that an $R$-module homomorphism $\varphi:L\to M$ induces a
surjective $R/I$-module homomorphism
\[
 L/IL\longrightarrow M/IM.
\]
Then $\varphi\otimes_R1_K$ is an isomorphism of $K$-vector spaces.
If, in addition, $k$ is a field, $A$ is a $k$-algebra, $R$ is a commutative $k$-
algebra, $L$ and $M$ are
$A\otimes_kR$-modules, and $\varphi$ is an $A\otimes_kR$-module homomorphism,
then $\varphi\otimes_R1_K$ is an $A\otimes_kK$-module isomorphism.
\end{lemma}

\begin{proof}
Let $C=\coker\varphi$, regarded as a finitely generated $R$-module.
The right exactness of taking quotients gives
\[
 L/IL\longrightarrow M/IM\longrightarrow C/IC\longrightarrow0.
\]
Hence $C=IC$. Choose a maximal ideal $\mathfrak m$ of $R$ containing
$I$. Then
\[
 C_{\mathfrak m}=(IR_{\mathfrak m})C_{\mathfrak m}
                 =(IR_{\mathfrak m})^nC_{\mathfrak m}
 \qquad(n\geq0).
\]
Lemma~\ref{krull} gives $C_{\mathfrak m}=0$. As $R$ is an integral
domain, $K$ is also the fraction field of $R_{\mathfrak m}$, and
\[
 C\otimes_RK
 \cong C_{\mathfrak m}\otimes_{R_{\mathfrak m}}K=0.
\]
Tensoring $L\to M\to C\to0$ with $K$ now shows that
$\varphi\otimes_R1_K$ is surjective. The two vector spaces have
the same finite dimension, so it is an isomorphism.
When the additional $A$-action is present, the tensor map is
an $A\otimes_kK$-module homomorphism, and so is its inverse.
\end{proof}

\begin{lemma}\label{c2:lem:finite-exceptional-elements}
Let $R$ be a commutative ring and let $E$ be a Noetherian
$R$-module. Suppose that a family
$\{a_\lambda\}_{\lambda\in\Lambda}\subseteq R$ satisfies
\[
 a_\lambda R+a_\mu R=R
 \qquad(\lambda\neq\mu).
\]
Then multiplication by $a_\lambda$ on $E$ is injective for all
but finitely many $\lambda\in\Lambda$.
\end{lemma}

\begin{proof}
For $\lambda\in\Lambda$, put
\[
 E_\lambda=\ker(a_\lambda:E\longrightarrow E),
 \qquad
 H=\sum_{\lambda\in\Lambda}E_\lambda.
\]
Since $E$ is a Noetherian $R$-module, its submodule $H$ is
finitely generated. Each member of a finite generating set
for $H$ belongs to a finite sum of the $E_\lambda$.
Consequently, there is a finite subset
$\Lambda_0\subseteq\Lambda$ such that
\[
 H=\sum_{\mu\in\Lambda_0}E_\mu.
\]
Set
\[
 d=\prod_{\mu\in\Lambda_0}a_\mu.
\]
If $H=0$, take $\Lambda_0=\varnothing$ and $d=1$.
Since $a_\mu E_\mu=0$ for every $\mu\in\Lambda_0$, we have
$dH=0$.

Let $\lambda\notin\Lambda_0$. In the quotient ring
$R/a_\lambda R$, each $a_\mu$, for $\mu\in\Lambda_0$, is
invertible because $a_\lambda R+a_\mu R=R$.
Their product $d$ is therefore invertible in this quotient,
so
$a_\lambda R+dR=R.$
As $a_\lambda E_\lambda=0$ and $dE_\lambda=0$, it follows that
\[
 E_\lambda
 =RE_\lambda
 =(a_\lambda R+dR)E_\lambda
 =0.
\]
Thus multiplication by $a_\lambda$ on $E$ is injective.
\end{proof}

\begin{lemma}\label{c2:lem:normalization-over-domain}
Let $k$ be a field, let $A$ be a $k$-algebra, and let $R$ be
a commutative Noetherian $k$-algebra which is an integral domain.
Let $0\neq a\in R$ be a nonunit. Let $F$ be an
$A\otimes_kR$-module which is finitely generated and torsion-free
as an $R$-module.

If $V$ is an $A$-module and $w:V\to F$ is a nonzero
$A$-module homomorphism, then there exist an integer $m\geq0$
and an $A$-module homomorphism $v:V\to F$ such that
\[
 (x)w=a^m((x)v)\qquad(x\in V),
 \qquad
 (V)v\nsubseteq aF.
\]
Moreover, if $V$ is also a left $A\otimes_kR$-module and $w$ is
an $A\otimes_kR$-module homomorphism, then so is $v$.
\end{lemma}

\begin{proof}
Since $a$ is not a unit, there is a maximal ideal
$\mathfrak m$ of $R$ containing $a$. The ring $R_{\mathfrak m}$
is Noetherian and local, $aR_{\mathfrak m}$ is a proper ideal,
and $F_{\mathfrak m}$
is a finitely generated $R_{\mathfrak m}$-module.
By Lemma~\ref{krull}, we have
\[
 \bigcap_{n\geq0}a^nF_{\mathfrak m}=0.
\]
Since $R$ is an integral domain and $F$ is torsion-free,
the natural map $f:F\to F_{\mathfrak m}$ is injective.
Since $(\bigcap_{n\geq0}a^nF)f=0$,
\[
 \bigcap_{n\geq0}a^nF=0.
\]

As $w\neq0$, its image cannot be contained in every $a^nF$.
Therefore there is a largest integer $m\geq0$ such that
$(V)w\subseteq a^mF.$
Consider multiplication by $a^m$,
\[
 \mu_m:F\longrightarrow a^mF,
 \qquad (x)\mu_m=a^mx.
\]
This map is an
$A\otimes_kR$-module isomorphism because
$a^m\neq0$ and $F$ is torsion-free over $R$.

Since $(V)w\subseteq a^mF$, we can define an $A$-module homomorphism
\[
 v=w\mu_m^{-1}:V\longrightarrow F.
\]
Then
\[
 (x)w=a^m((x)v)\qquad(x\in V).
\]
Assume $(V)v\subseteq aF$. Then
\[
 (V)w=a^m((V)v)\subseteq a^{m+1}F,
\]
contrary to the maximality of $m$. Hence
$(V)v\nsubseteq aF$.
\end{proof}

\subsection{Graded algebras}\label{sec:prelim-graded}

We collect the graded-module results needed in the infinite-field case.
They relate composition multiplicities and simple self-extensions in
the graded category to the corresponding ungraded data.
For the general theory, see \cite{GordonGreen1982,CamilloFuller,Dugas}.

We use the following facts for finite-dimensional algebras graded
by a free abelian group of finite rank: the forgetful functor
preserves and reflects indecomposability, and two indecomposable
graded modules with isomorphic underlying ungraded modules differ by a
unique shift; see \cite[Corollary~1.3 and Theorem~1.5]{Farnsteiner2009}.

\begin{definition}
Let $G$ be a group, and let $A=\bigoplus_{g\in G}A_g$ be a $G$-graded $k$-algebra. Let $M=\bigoplus_{g\in G}M_g$ be a $G$-graded $A$-module. Define
\[
 \operatorname{supp}(M)
 :=
 \{g\in G\mid M_g\ne0\},
\]
and call it the \emph{support of $M$ with respect to the grading}.

Let $A\modu_{\mathrm{gr}}^G$ denote the category
of finite-dimensional $G$-graded $A$-modules with degree-preserving $A$-module homomorphisms.
For a subset $I\subseteq G$, let
$A\modu_{\mathrm{gr}}^I$ denote the full
subcategory of $A\modu_{\mathrm{gr}}^G$
consisting of the modules $M$ with
$\operatorname{supp}(M)\subseteq I$.
For $G=\mathbb Z^r$ and $a\in\mathbb Z^r$, our shift convention is
\[
 M[a]_v=M_{v+a}\qquad(v\in\mathbb Z^r).
\]
Thus $\operatorname{supp}(M[a])=\operatorname{supp}(M)-a$.
Shifts are exact autoequivalences of the graded module category
and do not change the underlying ungraded module.
\end{definition}

We first record a criterion ensuring that the radical is homogeneous
and the quotient is concentrated in the identity degree.

\begin{lemma}\label{c2:prop:graded-radical-reduced-quotient}
Let \(G\) be a torsion-free group with identity \(e\), and let $
R=\bigoplus_{g\in G}R_g$
be a \(G\)-graded ring. Assume that
$\operatorname{supp}(R)=\{g\in G\mid R_g\ne 0\}$
is finite and that \(R/J(R)\) has no nonzero nilpotent elements.
Then
\[
J(R)=J(R_e)\oplus\bigoplus_{g\ne e}R_g,
\qquad
R/J(R)\cong R_e/J(R_e),
\]
where the isomorphism is induced by the inclusion
\(R_e\hookrightarrow R\).
In particular, \(J(R)\) is homogeneous, and the induced
\(G\)-grading on \(R/J(R)\) is concentrated in degree \(e\).
\end{lemma}

\begin{proof}
Put \(J=J(R)\). Let \(g\ne e\) and \(x\in R_g\).
Since \(G\) is torsion-free, the elements
\(g,g^2,g^3,\ldots\) are pairwise distinct.
As \(\operatorname{supp}(R)\) is finite, there exists
\(n\ge 1\) such that \(R_{g^n}=0\).
Hence \(x^n\in R_{g^n}=0\).
Thus \(x+J\) is nilpotent in \(R/J\), so the hypothesis
on \(R/J\) gives \(x\in J\). Consequently,
\[
R_g\subseteq J\qquad(g\ne e).
\]
Taking homogeneous components now yields
\[
J=(J\cap R_e)\oplus\bigoplus_{g\ne e}R_g.
\]
In particular, \(J\) is homogeneous, and the composite
\[
R_e\hookrightarrow R\twoheadrightarrow R/J
\]
is surjective with kernel \(J\cap R_e\).

It remains to show that \(J\cap R_e=J(R_e)\).
First, every element of \(R_e\) that is invertible in \(R\)
is already invertible in \(R_e\).
Indeed, if \(u\in R_e\) and \(v\in R\) satisfy \(uv=vu=1\),
then taking degree-\(e\) components gives
\(uv_e=v_eu=1\), where \(v_e\in R_e\) is the degree-\(e\)
component of \(v\).
Now let \(a\in J\cap R_e\). For every \(b\in R_e\),
the element \(1-ba\) is invertible in \(R\), hence in \(R_e\).
By the usual characterization of the Jacobson radical,
\(a\in J(R_e)\). Thus \(J\cap R_e\subseteq J(R_e)\).

Conversely, the surjection \(R_e\twoheadrightarrow R/J\)
maps \(J(R_e)\) into \(J(R/J)=0\).
Therefore \(J(R_e)\subseteq J\cap R_e\), and equality follows.
The asserted decomposition and isomorphism follow.
Since \(R_g\subseteq J\) for \(g\ne e\), the quotient grading
is concentrated in degree \(e\).
\end{proof}

\begin{lemma}
\label{c2:lem:graded-composition-factors}
Let $k$ be a field, and let $\Lambda$ be a finite-dimensional
$\mathbb Z^r$-graded $k$-algebra satisfying $\X$.
Then any two indecomposable graded $\Lambda$-modules with the same
graded composition multiplicities are isomorphic as graded modules.
The same assertion holds in
$\Lambda\modu_{\mathrm{gr}}^I$ for every subset
$I\subseteq\mathbb Z^r$.
\end{lemma}

\begin{proof}
Let $M$ and $N$ be indecomposable graded modules with the same
graded composition multiplicities, and let $U:\Lambda\modu_{\mathrm{gr}}^{\mathbb Z^r}\to \Lambda\modu$ denote the
forgetful functor. By
\cite[Corollary~1.3]{Farnsteiner2009}, the underlying ungraded modules
$U(M)$ and $U(N)$ are indecomposable. Since $U$ is exact,
the equality of the graded composition multiplicities implies
that $U(M)$ and $U(N)$ have the same composition factors as $\Lambda$-modules. Thus
$U(M)\cong U(N)$ by property $\X$. By
\cite[Theorem~1.5]{Farnsteiner2009}, there exists
$d\in\mathbb Z^r$ such that $M\cong N[d]$ as graded modules.

Taking the component of any fixed degree is exact. Hence the
equality of the graded composition multiplicities also gives
$\dim_k M_g=\dim_k N_g$ for every $g\in\mathbb Z^r$.
In particular,
$S:=\operatorname{supp}(M)=\operatorname{supp}(N).$
The graded isomorphism $M\cong N[d]$ implies $S=S-d$.
Since $S$ is finite and nonempty and $\mathbb Z^r$ is torsion-free, $d=0$, and therefore
$M\cong N$ as graded modules.

Note that $\Lambda\modu_{\mathrm{gr}}^I$ is a Serre subcategory of $\Lambda\modu_{\mathrm{gr}}^{\mathbb Z^r}$. This proves the last assertion.
\end{proof}

Let $k$ be a field, let $\Lambda$ be a finite-dimensional basic
$\mathbb Z^r$-graded $k$-algebra, and put $J=J(\Lambda)$.
By Lemma~\ref{c2:prop:graded-radical-reduced-quotient},
the ideal $J$ is homogeneous and $\Lambda/J$ is concentrated
in degree zero. Thus the natural map
$\Lambda_0\to\Lambda/J$ is surjective with nilpotent kernel
$J_0:=J\cap\Lambda_0$.

A complete set of pairwise orthogonal primitive idempotents
of $\Lambda/J$ therefore lifts to $\Lambda_0$.
Since their images in $\Lambda/J$ are primitive, the lifted
idempotents are also primitive in $\Lambda$.
Fix such a set $f_1,\ldots,f_s\in\Lambda_0$.
For $1\leq i\leq s$, put
$C_i=f_i\Lambda f_i/f_iJf_i$ and
$T_i=\Lambda f_i/Jf_i$.

\begin{lemma}
\label{c2:lem:graded-simple-self-extensions}
For $v\in\mathbb Z^r$, let $T_{i,v}$ denote $T_i$ placed in
degree $v$. Then the simple graded $\Lambda$-modules are precisely
the modules $T_{i,v}$, and there are isomorphisms of $k$-vector spaces
\begin{equation}\label{c2:eq:graded-simple-self-ext}
 \operatorname{Ext}^1_{\mathrm{gr}\Lambda}(T_{i,v},T_{i,v})
 \cong
 \operatorname{Hom}_{C_i}
 \bigl(f_i(J/J^2)_0f_i,C_i\bigr).
\end{equation}
In particular,
\[
 \operatorname{Ext}^1_{\mathrm{gr}\Lambda}(T_{i,v},T_{i,v})=0
 \quad\Longleftrightarrow\quad
 f_i(J/J^2)_0f_i=0.
\]
These equivalent conditions hold whenever $f_i\Lambda_0f_i$
is a division ring. If $v\in I\subseteq\mathbb Z^r$, the same
formula holds with the extension group computed in
$\Lambda\modu_{\mathrm{gr}}^I$.
\end{lemma}

\begin{proof}
Let $T$ be a simple graded module. Since $J$ is homogeneous,
$JT$ is a graded submodule of $T$. The nilpotence of $J$
excludes $JT=T$, so $JT=0$. Thus $T$ is a graded
$\Lambda/J$-module. Since $\Lambda/J$ is concentrated in
degree zero, every homogeneous component of $T$ is a submodule.
Simplicity forces $T$ to be concentrated in a single degree,
where it is isomorphic to some $T_i$. Conversely, every
$T_{i,v}$ is simple.

Since shifts are exact autoequivalences, it suffices to prove
\eqref{c2:eq:graded-simple-self-ext} for $v=0$.
In the full graded module category, consider the exact sequence
\[
 0\longrightarrow Jf_i\longrightarrow\Lambda f_i
 \longrightarrow T_i\longrightarrow0.
\]
Every homomorphism $\Lambda f_i\to T_i$ vanishes on $Jf_i$,
because $JT_i=0$. Applying
$\operatorname{Hom}_{\mathrm{gr}\Lambda}(-,T_i)$ therefore gives
\[
 \operatorname{Ext}^1_{\mathrm{gr}\Lambda}(T_i,T_i)
 \cong
 \operatorname{Hom}_{\mathrm{gr}\Lambda}(Jf_i,T_i).
\]
Every map on the right vanishes on $J^2f_i$. Its induced map
from $(J/J^2)f_i$ can be nonzero only in degree zero.
Consequently,
\begin{align*}
 \operatorname{Hom}_{\mathrm{gr}\Lambda}(Jf_i,T_i)
 &\cong
 \operatorname{Hom}_{\Lambda/J}
 \bigl((J/J^2)_0f_i,T_i\bigr)\\
 &\cong
 \operatorname{Hom}_{C_i}
 \bigl(f_i(J/J^2)_0f_i,C_i\bigr).
\end{align*}
For the last isomorphism, we use that $\Lambda$ is basic,
so $\Lambda/J\cong\prod_{j=1}^sC_j$ and $T_i$ corresponds
to the regular $C_i$-module. This proves
\eqref{c2:eq:graded-simple-self-ext}.

If $f_i\Lambda_0f_i$ is a division ring, we have
\[
 f_iJ_0f_i=0
 \quad\text{and hence}\quad
 f_i(J/J^2)_0f_i=0.
\]

The last assertion follows from the fact that $\Lambda\modu_{\mathrm{gr}}^I$ is a Serre subcategory of $\Lambda\modu_{\mathrm{gr}}^{\mathbb Z^r}$.
\end{proof}

\section{Reduction to finite-dimensional algebras}
\label{c2:sec:replacement}\label{c2:sec:structure}

We first reduce the central base ring to a local one. If
$R=R^{(1)}\times\cdots\times R^{(u)}$ is the decomposition of the
commutative artinian base into local rings, its central idempotents give
$A=A^{(1)}\times\cdots\times A^{(u)}$, omitting zero factors.
Every indecomposable $A$-module is supported on exactly one factor.
Each $A^{(v)}$ is a quotient of $A$, and the finite product has
finite representation type if and only if each factor has finite
representation type. Hence it suffices to work over a fixed
commutative artinian local base $(R,\mathfrak m,k)$.

Throughout the structural reduction, $A$ is finite over this $R$ and
$J=\rad A$. Let $E=E_R(k)$ be the injective envelope of its residue
field, and put $\D=\Hom_R(-,E)$. We use $\ell_R$ for length over the
commutative base. The division rings at primitive vertices will be
written $D_i$, distinct from the duality symbol $\D$.

Note that property $\X$ is invariant under Morita equivalence and passage to the
opposite algebra, and passes to quotient algebras. Representation-finiteness
is invariant under Morita equivalence and passage to the opposite algebra.

We may assume $A$ is basic.
Fix a complete set $e_1,\ldots,e_s$ of pairwise orthogonal primitive
idempotents. Write $S_i$ for the corresponding simple left module,
$A_i=e_iAe_i$, $\mathfrak r_i=e_iJe_i=\rad A_i$, and
$D_i=A_i/\mathfrak r_i$.

The next three lemmas describe the restrictions imposed by $\X$ on
primitive corners and their interaction with the off-diagonal components.

\begin{lemma}\label{c2:lem:equal-top-socle}
Suppose that $A$ satisfies $\X$ and that $0\neq M\in A\modu$.

$(1)$ If $M$ has simple top and simple
socle, both isomorphic to $S$, then $M$ is uniserial and every composition
factor of $M$ is isomorphic to $S$.

$(2)$ If $A$ is a local $R$-algebra, then $A$ is a self-injective Nakayama algebra. In that case, if $J\ne0$, there is
$t\in J$ such that
$J^a=At^a=t^aA$ for every $a\geq0$.
\end{lemma}

\begin{proof}
$(1)$ Suppose that $M$ is not simple. Then $\rad M$ and $M/\soc M$ are indecomposable $A$-modules and have
the same composition factors. Property $\X$ gives $\rad M \cong M/\soc M$. In particular, $\rad M$ has
simple top isomorphic to $S$. Induction on the length of $M$ now shows that $\rad M$ is uniserial
and that all its composition
factors are isomorphic to $S$. Since every proper submodule of $M$ is contained in $\rad M$ and
$M/\rad M \cong S$, the same
assertions hold for $M$.

$(2)$ Since $A$ is local, the indecomposable modules $A$ and ${\D}(A_A)$ have the same composition factors. Thus $A\cong  \D(A_A)$ by $\X$. Consequently, $A$
is self-injective and its socle is simple. Applying $(1)$, and then the analogous argument to the
opposite algebra, shows that $A$ is a local Nakayama algebra. Any $t$ in $J$ outside $J^2$ generates $J$
on both sides, and the assertion about its powers follows.
\end{proof}

\begin{lemma}\label{c2:lem:no-return}
Let $A$ be basic and satisfy $\X$, and choose pairwise orthogonal
primitive idempotents $e_1,\ldots,e_s$ with sum $1$. For $i\ne j$,
\[
 e_iAe_jAe_i=0.
\]
Consequently $A_i=e_iAe_i$ is a quotient algebra of $A$ and satisfies
$\X$.
\end{lemma}
\begin{proof}
Suppose that $0\ne x\in e_iAe_jAe_i$. There is an $R$-module homomorphism
$\lambda:e_iAe_i\to E$ with $(x)\lambda\ne0$. The formula
\[
 f:Ae_i\longrightarrow\D(e_iA),\qquad
 (b)((p)f)=(bp)\lambda
\]
defines an $A$-module homomorphism, and $(e_i)((x)f)\ne0$.
Put $M=\im f$. Then $\topm M\cong\soc M\cong S_i$ as $A$-modules.
Lemma~\ref{c2:lem:equal-top-socle} implies
that all composition factors of $M$ are isomorphic to $S_i$. Basicness and $j\ne i$ therefore give
$e_jM=0$. Writing $x$ as a sum of products $ab$ with
$a\in e_iAe_j$ and $b\in e_jAe_i$ now gives
$(x)f=\sum a((b)f)=0$, a contradiction.

The map $A\to A_i$, $a\mapsto e_iae_i$, is multiplicative: the omitted
terms in $e_iab e_i$ factor through some $e_j$ with $j\ne i$, and vanish
by what was proved. It is a unital epimorphism when the identity of
$A_i$ is taken to be $e_i$. Quotient permanence proves the last assertion.
\end{proof}

\begin{lemma}\label{2.3}
Let $A$ be a basic Artin algebra satisfying $\X$. Then
\[
 \mathfrak r_i(e_iAe_j)=0=(e_iAe_j)\mathfrak r_j
 \qquad(i\ne j).
\]
\end{lemma}

\begin{proof}
Since property $\X$ is invariant under passage to the
opposite algebra, it suffices to prove
$\mathfrak r_i(e_iAe_j)=0$ for $i\ne j$.

Suppose that
$\mathfrak r_i(e_iAe_j)\ne0$.
We derive a contradiction in two steps. First, we construct
$A$-modules $M$ and $N$ such that $M$ is indecomposable and $N$
has the same composition factors as $M$ but is not isomorphic
to it.

By Lemma~\ref{c2:lem:no-return},
\[
 \pi: A\longrightarrow e_iAe_i,\qquad a\mapsto e_iae_i
\]
is a surjective ring homomorphism. Hence $e_iAe_i$ satisfies $\X$
and is a local Nakayama algebra by Lemma~\ref{c2:lem:equal-top-socle}(2).

Let
\[
 U=A(1-e_i)Ae_iAe_j,\qquad P_0=Ae_j/U.
\]
Since $A$ is basic and $i\ne j$, we have
$e_iAe_j=e_iJe_j\subseteq Je_j$, and hence $U\subseteq Je_j$.
Thus $\operatorname{top}P_0\cong S_j$ as $A$-modules.
Lemma~\ref{c2:lem:no-return} also gives
\[
 e_iU=e_iA(1-e_i)Ae_iAe_j=0,
\]
so $e_iP_0\cong e_iAe_j$ as $e_iAe_i$-modules.
Since
$(1-e_i)Ae_iP_0=0$, $e_iP_0$ is an $A$-submodule of $P_0$, and its action is
given by $\pi$.

Since $\mathfrak r_i(e_iP_0)\ne0$, the $A_i$-module $e_iP_0$ has an indecomposable direct summand of length
at least two. Thus there is an epimorphism
\[
 \varphi:e_iP_0\twoheadrightarrow
 T:=(e_iAe_i)/\mathfrak r_i^2.
\]
Regard $T$ as an $A$-module via $\pi$.
Put $K=\ker\varphi$ and $M=P_0/K$.
The inclusion $e_iP_0\subseteq P_0$ induces an embedding
\[
 \lambda: T\cong e_iP_0/K\hookrightarrow P_0/K=M
\]
whose image is $e_iM$. Since $K\subseteq e_iP_0\subseteq JP_0$,
we have $\operatorname{top}M\cong S_j$, so $M$ is indecomposable.
We identify $T$ with $e_iM\subseteq JM$.
As an $A$-module, $T$ is uniserial of length two, and every composition factor of $T$ is isomorphic to $S_i$.

Consider the pushout diagram in $A\modu$:
\[
\begin{tikzcd}[column sep=large, row sep=large]
0 \arrow{r}
  & T/S_i \arrow[hook]{r}{\bar{\lambda}} \arrow[hook]{d}{\theta}
  & M/S_i \arrow{r}{q} \arrow[hook]{d}{\alpha}
  & M/T \arrow{r} \arrow[double, double distance=1pt, -]{d}
  & 0 \\
0 \arrow{r}
  & T \arrow[hook]{r}{\beta}
  & N \arrow{r}{\rho}
  & M/T \arrow{r}
  & 0 ,
\end{tikzcd}
\]
where $\bar{\lambda}$ is induced by $\lambda$ and $\theta$ is the canonical inclusion. This gives the following exact sequence in $A\modu$:
\[
 0\longrightarrow T/S_i
 \xrightarrow{(\bar{\lambda},-\theta)}(M/S_i)\oplus T
 \longrightarrow N\longrightarrow0.
\]
Thus $M$ and $N$ have the same composition factors.
Since $\im (\bar{\lambda},-\theta)$ is contained in $J((M/S_i)\oplus T)$, $\topm(N)\cong S_j\oplus S_i$ as $A$-modules. Therefore $M\not\cong N$ as $A$-modules.

Second, we prove that $N$ is an indecomposable $A$-module, which gives a contradiction to property $\X$.

We note that $\im\alpha\cong M/S_i$ and
$\im\beta\cong T$ are indecomposable $A$-modules, and
\[
 \im\alpha\cap\im\beta
 =(\ker q)\alpha
 =(T/S_i)\bar{\lambda}\alpha
 =(T/S_i)\theta\beta
 =(S_i)\beta\ne0.
\]
Since $\Hom_A(Ae_i,\bar{\lambda})$ is an isomorphism and $e_iT=T$, $e_iN=\im\beta$.
Note that $M/S_i$ is a quotient of $Ae_j$ and $\Hom_A(Ae_j,\alpha)$ is an isomorphism. It follows that $A(e_jN)=\im\alpha$.
Thus both $\im\alpha$ and $\im\beta$ are preserved by every $A$-module endomorphism of $N$. In particular, every idempotent endomorphism of $N$ preserves them.

Let $h$ be an idempotent endomorphism of $N$.
Its restriction to each of these indecomposable submodules is
either zero or the identity. Their intersection is nonzero,
so the two restrictions are either both zero or both identities.
Since $\im\alpha+\im\beta=N$,
we obtain $h=0$ or $h=1_N$.
Hence $N$ is indecomposable. This completes the proof.
\end{proof}

We now use these structural restrictions to construct the
finite-dimensional algebra required for the reduction.
Retain the commutative artinian local base $R$ fixed above, with residue
field $k=R/\mathfrak m$. The construction replaces each primitive corner
while preserving the off-diagonal data. The resulting correspondence
on module isomorphism classes preserves direct sums and composition
multiplicities.

We first recall the definition of a \emph{skew polynomial algebra}. Let $D$ be a division $k$-algebra and let $\sigma$ be a $k$-algebra endomorphism of $D$.
The \emph{skew polynomial algebra} $D[X;\sigma]$ consists of all formal
finite sums
\[
 \sum_{m\geq 0} a_m X^m,\qquad a_m\in D,
\]
with coefficientwise addition and multiplication determined by
\[
 (a X^m)(b X^n)
 =a\bigl((b)\sigma^m\bigr)X^{m+n}
 \qquad(a,b\in D,\ m,n\geq 0),
\]
where $\sigma^0=\mathrm{id}_D$.

\begin{lemma}\label{c2:lem:replacement}

Let $A$ be a basic Artin $R$-algebra satisfying $\X$.
There exists a finite-dimensional basic $k$-algebra $A^0$ and a
bijection between the isomorphism classes of finitely generated
left $A$-modules and those of finitely generated left $A^0$-modules
which is compatible with direct sums and preserves composition
multiplicities under a bijection between the isomorphism classes
of simple modules.
In particular, $A^0$ satisfies $\X$, and $A$ is representation-finite
if and only if $A^0$ is representation-finite.
\end{lemma}

\begin{proof}
If $A=0$, take $A^0=0$. Suppose that $A\ne0$.
We divide the proof into four steps.

\smallskip
\noindent\emph{(1) Construction of $A^0$.}

Retain the notation fixed above. Let $\ell_i$ be the Loewy length of $A_i$.
By Lemmas~\ref{c2:lem:equal-top-socle}
and~\ref{c2:lem:no-return}, each $A_i$ is a local Nakayama algebra.
Moreover, Lemmas~\ref{c2:lem:no-return} and~\ref{2.3} give
\begin{equation}\label{c2:eq:cleft-annihilation}
 e_iAe_jAe_i=0,\qquad
 \mathfrak r_ie_iAe_j=0=e_iAe_j\mathfrak r_j
 \qquad(i\ne j).
\end{equation}

Set
\[
 B=\prod_{i=1}^sA_i,\qquad D=\prod_{i=1}^sD_i,
 \qquad H=\bigoplus_{i\ne j}e_iAe_j,
\]
and let $\pi:B\twoheadrightarrow D$ be the natural epimorphism.
The map $(b_i)_i\mapsto\sum_i b_i$ identifies $B$ with a unital
subalgebra of $A$. Clearly, $A=B\oplus H$
as $R$-modules and $H$ is a two-sided ideal
of $A$. Since $A$ is basic, $e_iAe_j=e_iJe_j$ for $i\neq j$. Thus $H\subset J$, so $H$ is nilpotent. The equalities \eqref{c2:eq:cleft-annihilation} give
$J(B)H=0=HJ(B).$
Consequently $H$ is a $D$-$D$-bimodule, and multiplication induces a $D$-$D$-bimodule homomorphism
\[
 \mu:H\otimes_DH\longrightarrow H,
 \qquad (h\otimes h')\mu=hh'.
\]
Under the decomposition
$A=B\oplus H$, the multiplication of $A$ is
\begin{equation*}
 (b,h)(b',h')
 =\bigl(bb',\ (b)\pi h'+h(b')\pi+(h\otimes h')\mu\bigr).
\end{equation*}

Since $\mathfrak m A_i$ is a nilpotent ideal of $A_i$, we have
$\mathfrak m A_i\subseteq\mathfrak r_i$ and
$\mathfrak m\mathfrak r_i\subseteq\mathfrak r_i^2$.
The image of $R$ in $A$ lies in $B$, and the image of
$\mathfrak m$ lies in $J(B)$. Hence $\mathfrak mH=0$.
It follows that $D_i$, $H$ and $\mathfrak r_i/\mathfrak r_i^2$
are finite-dimensional $k$-vector spaces, and their left and right
scalar actions agree.

Suppose first that $\ell_i>1$. Then $\mathfrak r_i/\mathfrak r_i^2$ has dimension one as
both a left and a right $D_i$-vector space. By Lemma~\ref{c2:lem:equal-top-socle}, we can choose $0\ne\xi_i\in \mathfrak r_i/\mathfrak r_i^2$ such that $\xi_i(A_i/\mathfrak r_i^2)=(A_i/\mathfrak r_i^2)\xi_i$.
For every $d\in A_i$, there is an element $d'\in A_i$ such that $\xi_i \bar{d} =\bar{d'} \xi_i$ in $A_i/\mathfrak r_i^2$. Define $\sigma_i:D_i\to D_i$ by $(\bar{d})\sigma_i=\bar{d'}$. It is not difficult to see that $\sigma_i$ is a well-defined $k$-algebra automorphism of $D_i$.
Define $A_i^0=D_i[X_i;\sigma_i]/(X_i^{\ell_i})$. By \cite[Lemma~4.1]{AAS2003} and \cite[Theorem~2.1]{EG1971}, $A_i^0$ is a local Nakayama algebra of Loewy
length $\ell_i$. We denote $J(A_i^0)$ by $\mathfrak r_i^0$.
Note that $A_i^0/\mathfrak r_i^0\cong D_i$ as $k$-algebras and the map
\begin{equation}\label{c2:eq:cleft-first-layer}
 \mathfrak r_i/\mathfrak r_i^2\longrightarrow \mathfrak r_i^0/(\mathfrak r_i^0)^2,
 \qquad d\xi_i\longmapsto dX_i+(\mathfrak r_i^0)^2
\end{equation}
is a $D_i$-$D_i$-bimodule isomorphism.
For $\ell_i=1$, we set $A_i^0=D_i$.

Let $B^0=\prod_iA_i^0$, and let
$\pi^0:B^0\twoheadrightarrow D$ be the natural epimorphism. Regard $H$ as a $B^0$-bimodule via $\pi^0$. Then
\begin{equation}\label{sec5u}
J(B^0)H=0=HJ(B^0).
\end{equation}
Define $A^0$ on the $k$-vector space $B^0\oplus H$ by
\begin{equation}\label{c2:eq:cleft-new-product}
 (b,h)(b',h')
 =\bigl(bb',\ (b)\pi^0 h'+h(b')\pi^0+(h\otimes h')\mu\bigr).
\end{equation}
Thus $A^0$ is a finite-dimensional $k$-algebra with identity $(1,0)$ and
\begin{equation}\label{c2:eq:cleft-new-radical}
 J(A^0)=J(B^0)\oplus H,\qquad A^0/J(A^0)\cong D.
\end{equation}
In particular, $A^0$ is basic.
Let $\varepsilon_i\in B^0$ be the identity of the $i$th factor,
and put $e_i^0=(\varepsilon_i,0)\in A^0$.
Then
\begin{equation}\label{sec5u2}
 e_i^0A^0e_i^0\cong A_i^0,\qquad
 e_i^0A^0e_j^0\cong e_iAe_j,\qquad e_i^0He_j^0He_i^0=0 \quad(i\ne j),\qquad e_i^0He_i^0=0.
\end{equation}

\smallskip

\noindent\emph{(2) Comparison of quotient categories.}

For an Artin algebra $\Lambda$, define an ideal of the additive category $\Lambda\modu$ by
\[
 \mathcal N_\Lambda(V,V')
 =\{f\in\Hom_\Lambda(V,V'):
       \topm_\Lambda(f)=0=\soc_\Lambda(f)\}.
\]
Write $\mathcal E_\Lambda=(\Lambda\modu)/\mathcal N_\Lambda$, and let $q_\Lambda:\Lambda\modu\longrightarrow\mathcal E_\Lambda$ denote the quotient functor.
The quotient has the same objects as $\Lambda\modu$, with
\[
 \Hom_{\mathcal E_\Lambda}(V,V')
 =
 \Hom_\Lambda(V,V')/\mathcal N_\Lambda(V,V').
\]
Denote the class of $f$ by $[f]$. Thus $[f]=[g]$ precisely
when $f$ and $g$ induce the same maps on both tops and socles.
Moreover, if $J(\Lambda)^d=0$, then $\mathcal N_\Lambda^d=0$:
for $f\in\mathcal N_\Lambda(V,V')$,
\[
 (J(\Lambda)^aV)f\subseteq J(\Lambda)^{a+1}V'
 \qquad(a\geq0),
\]
and the assertion follows by iteration.

We first construct an additive equivalence
\[
 F_i:\mathcal E_{A_i}\longrightarrow\mathcal E_{A_i^0}
\]
for each $i$.
Fix $i$. By Step~(1), $A_i$ and $A_i^0$ are local Nakayama
algebras of Loewy length $\ell_i$. Their indecomposable
finitely generated modules are therefore
\[
 U_p=A_i/\mathfrak r_i^p,\qquad
 U_p^0=A_i^0/(\mathfrak r_i^0)^p
 \qquad(1\leq p\leq\ell_i),
\]
respectively.
On these objects define
\begin{equation*}
 F_i(U_p)=U_p^0\qquad(1\leq p\leq\ell_i).
\end{equation*}
If $\ell_i=1$, take $F_i$ to be the identity under the
identifications $A_i\cong D_i\cong A_i^0$.
Suppose that $\ell_i>1$.
Choose a lift $x_i\in\mathfrak r_i$ of $\xi_i$.
By Nakayama's lemma, $x_i$ has the generating property
in Lemma~\ref{c2:lem:equal-top-socle}(2):
\[
 \mathfrak r_i^a=A_ix_i^a=x_i^aA_i\qquad(a\geq0).
\]
We regard each $\mathfrak r_i^a/\mathfrak r_i^{a+1}$ as a
$D_i$-bimodule via the actions induced by $A_i$.
The definition of $\sigma_i$ in Step~(1) gives, by induction,
the following equality in $\mathfrak r_i^a/\mathfrak r_i^{a+1}$:
\begin{equation}\label{c2:eq:fi-layer-relation}
 (x_i^a+\mathfrak r_i^{a+1})d
 =\bigl((d)\sigma_i^a\bigr)
   (x_i^a+\mathfrak r_i^{a+1})
 \qquad(0\leq a<\ell_i,\ d\in D_i).
\end{equation}
The corresponding relation in $A_i^0$ is
$X_i^ad=(d)\sigma_i^a\,X_i^a$, consistently with
\eqref{c2:eq:cleft-first-layer}.

To define $F_i$ on morphisms, consider $f:U_p\to U_q$.
Write
\[
 (a+\mathfrak r_i^p)f=ab+\mathfrak r_i^q,\qquad
 m=\max\{q-p,0\}.
\]
As verified below, $b\in\mathfrak r_i^m$, and a unique
$c\in D_i$ is determined by the equality
\[
 b+\mathfrak r_i^{m+1}
 =c(x_i^m+\mathfrak r_i^{m+1})
 \quad\text{in }\mathfrak r_i^m/\mathfrak r_i^{m+1}.
\]
Associate to $f$ the homomorphism
\begin{equation}\label{c2:eq:fi-component}
 f^0:U_p^0\longrightarrow U_q^0,\qquad
 (a+(\mathfrak r_i^0)^p)f^0
 =ac\,X_i^m+(\mathfrak r_i^0)^q.
\end{equation}
This is well defined since $p+m\geq q$.
In particular, if $p=q$, then $f^0$ is right multiplication
by $c$.
Set
\begin{equation*}
 F_i([f])=[f^0].
\end{equation*}

We verify that this defines an equivalence.
Evaluation at $1+\mathfrak r_i^p$ gives an isomorphism
of additive groups
\begin{equation}\label{c2:eq:fi-hom}
 \Hom_{A_i}(U_p,U_q)
 \cong\mathfrak r_i^m/\mathfrak r_i^q.
\end{equation}
Indeed, the necessary and sufficient condition on $b$ is
$\mathfrak r_i^p b\subseteq\mathfrak r_i^q$.
It holds for every $b\in\mathfrak r_i^m$.
Conversely, if
$b\in\mathfrak r_i^j\setminus\mathfrak r_i^{j+1}$ with
$j<q-p$, write $b=x_i^ju$ with $u$ a unit.
Then $x_i^pb=x_i^{p+j}u\notin\mathfrak r_i^q$.
This proves \eqref{c2:eq:fi-hom} and justifies the coefficient
$c$ used above.

Identify the tops of $U_p$ and $U_p^0$ with $D_i$.
Their socles are $\mathfrak r_i^{p-1}/\mathfrak r_i^p$
and $(\mathfrak r_i^0)^{p-1}/(\mathfrak r_i^0)^p$,
respectively; identify these with $D_i$ by
\[
 d\longmapsto d(x_i^{p-1}+\mathfrak r_i^p),
 \qquad
 d\longmapsto d(X_i^{p-1}+(\mathfrak r_i^0)^p).
\]
Under these identifications,
\eqref{c2:eq:fi-layer-relation} yields, for $d\in D_i$,
\begin{equation*}
 \begin{array}{c|c|c}
      &(d)\topm_{A_i}(f)&(d)\soc_{A_i}(f)\\ \hline
 p<q&0&d\bigl((c)\sigma_i^{p-1}\bigr)\\
 p>q&dc&0\\
 p=q&dc&d\bigl((c)\sigma_i^{p-1}\bigr).
 \end{array}
\end{equation*}
For the socle calculation when $p\leq q$, the equality in
$\mathfrak r_i^{q-1}/\mathfrak r_i^q=\soc_{A_i}(U_q)$ is
\[
 x_i^{p-1}b+\mathfrak r_i^q
 =\bigl((c)\sigma_i^{p-1}\bigr)
   (x_i^{q-1}+\mathfrak r_i^q).
\]
When $p>q$, $x_i^{p-1}b+\mathfrak r_i^q=0$ in $U_q$.
The same calculation applies to homomorphisms $U_p^0\to U_q^0$.
Every such homomorphism induces the same maps on tops and
socles as a unique homomorphism $f^0$ of the form
\eqref{c2:eq:fi-component}.
Since $\sigma_i$ is an automorphism by Step~(1), it follows that
\[
 \mathcal N_{A_i}(U_p,U_q)
 \cong\mathfrak r_i^{m+1}/\mathfrak r_i^q
\]
under \eqref{c2:eq:fi-hom}.
Consequently $c$ depends only on $[f]$, and
$[f]\mapsto[f^0]$ gives an additive bijection
\[
 \Hom_{\mathcal E_{A_i}}(U_p,U_q)
 \xrightarrow{\sim}
 \Hom_{\mathcal E_{A_i^0}}(U_p^0,U_q^0).
\]
Since a morphism class is determined by its induced maps on tops and socles, the preceding identifications preserve identities and composition. Extending the construction to finite
direct sums by matrices gives a fully faithful additive functor $F_i$. Every finitely generated $A^0_i$-module is a finite direct sum of the modules $U_p^0$, so $F_i$
is dense and hence an equivalence.

Using the product decompositions from Step~(1),
we obtain an additive equivalence
\begin{equation*}
 E:\mathcal E_B\xrightarrow{\sim}\mathcal E_{B^0}.
\end{equation*}

For $\Lambda=B,B^0$, let
$\overline \topm_\Lambda:\mathcal E_\Lambda\to D\modu$
be induced by $\topm_\Lambda$, and let
$\iota_\Lambda:D\modu\to\mathcal E_\Lambda$ be restriction
of scalars along $\Lambda\twoheadrightarrow D$ followed
by $q_\Lambda$.
It is fully faithful, since $\mathcal N_\Lambda$ vanishes
on homomorphisms with semisimple source.
The construction preserves the maps on tops and, on
semisimple modules, the matrices of $D_i$-module homomorphisms.
The additive extension therefore gives natural isomorphisms
\begin{equation}\label{c2:eq:cleft-functor-compatibility}
 \overline \topm_{B^0}E\cong\overline \topm_B,\qquad
 E\iota_B\cong\iota_{B^0}.
\end{equation}
On semisimple modules these identifications are $D$-module isomorphisms.

\smallskip

\noindent\emph{(3) Comparison of module structures.}
For $\Lambda\in\{B,B^0\}$, put $F_\Lambda=H\otimes_\Lambda-$,
and let $\eta^\Lambda:F_\Lambda^2\to F_\Lambda$ be induced by
the multiplication on $H$:
\[
 (h\otimes(h'\otimes v))\eta^\Lambda_V=hh'\otimes v.
\]
Following \cite[\S2.2]{Beligiannis2000}, let
$\mathcal C_\Lambda$ be the additive category whose objects
are pairs $(V,\alpha)$, where $V\in\Lambda\modu$ and
$\alpha:F_\Lambda V\to V$ satisfies
\begin{equation*}
 F_\Lambda(\alpha)\alpha=\eta^\Lambda_V\alpha,
\end{equation*}
and whose morphisms $(V,\alpha)\to(V',\alpha')$ are
$\Lambda$-module homomorphisms $f:V\to V'$ satisfying
\begin{equation*}
 \alpha f=F_\Lambda(f)\alpha'.
\end{equation*}
Composition, addition and finite direct sums are inherited
from $\Lambda\modu$.
The description of modules over cleft extensions in
\cite[\S2.3]{Beligiannis2000} gives additive equivalences
\begin{equation}\label{c2:eq:cleft-module-equivalences}
 A\modu\cong\mathcal C_B,
 \qquad A^0\modu\cong\mathcal C_{B^0}.
\end{equation}

Since $HJ(\Lambda)=0=J(\Lambda)H$, we have a natural
isomorphism
\begin{equation}\label{c2:eq:cleft-tensor-top}
 F_\Lambda V\cong H\otimes_D\topm_\Lambda V.
\end{equation}

Note that $F_\Lambda V$ is semisimple, so
$\soc_\Lambda(F_\Lambda V)=F_\Lambda V$.
Every homomorphism $g:F_\Lambda V\to V'$ has image in
$\soc_\Lambda V'$, so $\soc_\Lambda(g)$ is the same map
with codomain restricted to $\soc_\Lambda V'$.
In particular, $\soc_\Lambda(g)=0$ implies $g=0$.
By the definition of $\mathcal N_\Lambda$, this proves
\begin{equation}\label{c2:eq:cleft-vanishing-ideal}
 \mathcal N_\Lambda(F_\Lambda V,V')=0
 \qquad(V,V'\in\Lambda\modu).
\end{equation}
Since morphism groups in $\mathcal E_\Lambda$ are obtained
by quotienting by $\mathcal N_\Lambda$, $q_\Lambda$
therefore induces an isomorphism
\begin{equation}\label{c2:eq:cleft-unique-lift}
 \Hom_\Lambda(F_\Lambda V,V')
 \xrightarrow{\sim}
 \Hom_{\mathcal E_\Lambda}
       (q_\Lambda F_\Lambda V,q_\Lambda V').
\end{equation}

Since $F_\Lambda(\mathcal N_\Lambda)=0$, we have an
endofunctor $\overline F_\Lambda$ of $\mathcal E_\Lambda$:
\[
 \overline F_\Lambda(q_\Lambda V)=q_\Lambda F_\Lambda V,
 \qquad
 \overline F_\Lambda(q_\Lambda f)=q_\Lambda F_\Lambda(f).
\]
Applying $q_\Lambda$ to the components of $\eta^\Lambda$
gives $\overline\eta^\Lambda:
\overline F_\Lambda^2\to\overline F_\Lambda$.
Let $\overline{\mathcal C}_\Lambda$ have objects $(X,a)$,
where $X\in\mathcal E_\Lambda$ and
$a:\overline F_\Lambda X\to X$ satisfies
\[
 \overline F_\Lambda(a)a=\overline\eta^\Lambda_Xa.
\]
A morphism $(X,a)\to(Y,b)$ is a morphism $u:X\to Y$
in $\mathcal E_\Lambda$ satisfying
$au=\overline F_\Lambda(u)b$.
The quotient functor induces an additive functor
\[
 Q_\Lambda:\mathcal C_\Lambda\longrightarrow
              \overline{\mathcal C}_\Lambda,
 \qquad
 (V,\alpha)\longmapsto(q_\Lambda V,q_\Lambda\alpha),
 \quad f\longmapsto q_\Lambda f.
\]
By \eqref{c2:eq:cleft-unique-lift}, every structure map
$a:\overline F_\Lambda(q_\Lambda V)\to q_\Lambda V$
has a unique lift $\alpha:F_\Lambda V\to V$.
The difference
$F_\Lambda(\alpha)\alpha-\eta^\Lambda_V\alpha$
is annihilated by $q_\Lambda$ and has semisimple source
$F_\Lambda^2V$, so it is zero by
\eqref{c2:eq:cleft-vanishing-ideal}.
Similarly, a morphism between the images of
$(V,\alpha)$ and $(V',\alpha')$ has a representative
$f:V\to V'$, and
\[
 \alpha f-F_\Lambda(f)\alpha':F_\Lambda V\longrightarrow V'
\]
is annihilated by $q_\Lambda$ and hence is zero.
Thus $Q_\Lambda$ is full and dense.
Its kernel consists of the morphisms whose underlying
$\Lambda$-module homomorphisms belong to $\mathcal N_\Lambda$.
Consequently $(\ker Q_\Lambda)^L=0$ for
$L=\max_i\ell_i$, by Step~(2).

It remains to compare $\overline{\mathcal C}_B$ and
$\overline{\mathcal C}_{B^0}$.
Equation~\eqref{c2:eq:cleft-tensor-top} gives
\[
 \overline F_\Lambda
 \cong\iota_\Lambda(H\otimes_D-)\overline{\topm}_\Lambda.
\]
Use the identifications in
\eqref{c2:eq:cleft-functor-compatibility} coming from
Step~(2), chosen to agree on semisimple modules.
They yield a natural isomorphism
\[
 \tau:E\overline F_B\xrightarrow{\sim}\overline F_{B^0}E.
\]
Under these identifications, $E$ preserves the underlying
$D$-module homomorphisms between semisimple modules, and both
multiplications are induced by
\[
 \mu\otimes1:
 H\otimes_DH\otimes_D\overline{\topm}_B X
 \longrightarrow H\otimes_D\overline{\topm}_B X.
\]
Therefore
\begin{equation}\label{c2:eq:cleft-multiplication-compatibility}
 E(\overline\eta^B_X)\tau_X
 =\tau_{\overline F_BX}\,
   \overline F_{B^0}(\tau_X)\,
   \overline\eta^{B^0}_{EX}.
\end{equation}
For any morphism $a:\overline F_BX\to X$, put
$a^0=\tau_X^{-1}E(a)$.
Naturality of $\tau$ and
\eqref{c2:eq:cleft-multiplication-compatibility} give
\[
 \begin{aligned}
 \overline F_{B^0}(a^0)a^0-\overline\eta^{B^0}_{EX}a^0=\overline F_{B^0}(\tau_X^{-1})\,
 \tau_{\overline F_BX}^{-1}\,
 E\bigl(\overline F_B(a)a-\overline\eta^B_Xa\bigr).
 \end{aligned}
\]
Since $E$ is faithful and the two preceding factors are
invertible, $a^0$ satisfies the associativity relation
if and only if $a$ does.
Define
\[
 \widetilde E:\overline{\mathcal C}_B
       \longrightarrow\overline{\mathcal C}_{B^0},
 \qquad
 (X,a)\longmapsto(EX,a^0),\quad u\longmapsto E(u).
\]
For $u:X\to Y$ and $b:\overline F_BY\to Y$, put
$b^0=\tau_Y^{-1}E(b)$.
Naturality of $\tau$ also gives
\[
 a^0E(u)-\overline F_{B^0}(E(u))b^0
 =\tau_X^{-1}E\bigl(au-\overline F_B(u)b\bigr).
\]
Hence $E(u)$ is a morphism between the corresponding
objects if and only if $u$ is.
Since $E$ is fully faithful, so is $\widetilde E$;
addition, identities and composition are preserved.

To prove essential surjectivity, let $(Y,b)$ be an object of
$\overline{\mathcal C}_{B^0}$, and choose an isomorphism
$u:EX\to Y$.
Naturality of $\overline\eta^{B^0}$ and the relation for
$b$ show that
\[
 \overline F_{B^0}(u)b\,u^{-1}:
 \overline F_{B^0}EX\longrightarrow EX
\]
satisfies the associativity relation.
By full faithfulness of $E$, there is a unique
$a:\overline F_BX\to X$ such that
\[
 E(a)=\tau_X\,\overline F_{B^0}(u)b\,u^{-1}.
\]
The preceding calculation shows that $a$ satisfies the
associativity relation.
Moreover, $a^0u=\overline F_{B^0}(u)b$, so $u$ is an
isomorphism $\widetilde E(X,a)\to(Y,b)$.
Thus $\widetilde E$ is an additive equivalence, and
\begin{equation}\label{c2:eq:cleft-pair-equivalence}
 \overline{\mathcal C}_B\cong
 \overline{\mathcal C}_{B^0}.
\end{equation}

\smallskip
\noindent\emph{(4) Isomorphism classes and composition multiplicities.}

For an additive category $\mathcal A$, let
$\operatorname{Iso}(\mathcal A)$ denote the commutative
monoid of isomorphism classes under direct sum.
By Step~(3), $Q_\Lambda$ is full and dense
with nilpotent kernel, and hence induces an equivalence
\[
 \mathcal C_\Lambda/\ker Q_\Lambda
 \cong\overline{\mathcal C}_\Lambda
 \qquad(\Lambda=B,B^0).
\]
Applying \cite[Lemma~2.5.2]{Pirashvili2008} to
$\ker Q_\Lambda$, we obtain
\[
 \operatorname{Iso}(\mathcal C_\Lambda)
 \cong\operatorname{Iso}(\overline{\mathcal C}_\Lambda).
\]
Together with \eqref{c2:eq:cleft-module-equivalences} and
\eqref{c2:eq:cleft-pair-equivalence}, this yields a monoid
isomorphism
\begin{equation}\label{c2:eq:cleft-iso-monoid}
 \operatorname{Iso}(A\modu)
 \xrightarrow{\sim}\operatorname{Iso}(A^0\modu).
\end{equation}
It preserves and reflects indecomposability.
Consequently, $A$ is representation-finite if and only if
$A^0$ is representation-finite.

To compare composition multiplicities, put
\[
 S_i=Ae_i/J(A)e_i,\qquad
 S_i^0=A^0e_i^0/J(A^0)e_i^0
 \qquad(1\leq i\leq s).
\]
Let $M$ and $M^0$ represent corresponding classes under
\eqref{c2:eq:cleft-iso-monoid}.
The construction of $E$ and the lifting in Step~(3) give
\[
 \begin{aligned}
 {}_BM&\cong
 \bigoplus_{i=1}^s\bigoplus_{p=1}^{\ell_i}
       (A_i/\mathfrak r_i^p)^{n_{ip}},\\
 {}_{B^0}M^0&\cong
 \bigoplus_{i=1}^s\bigoplus_{p=1}^{\ell_i}
       \bigl(A_i^0/(\mathfrak r_i^0)^p\bigr)^{n_{ip}}
 \end{aligned}
\]
for the same nonnegative integers $n_{ip}$.
Indeed, the quotient functors and the lifting of the
$H$-action leave the underlying modules unchanged.

The exact functor $e_i(-):A\modu\to A_i\modu$ sends
$S_i$ to $D_i$ and the other simple modules to zero.
Applying it to a composition series of $M$, and arguing
similarly for $e_i^0(-)$, gives
\[
 [M:S_i]
 =\length_{A_i}(e_iM)
 =\sum_{p=1}^{\ell_i}p\,n_{ip}
 =\length_{A_i^0}(e_i^0M^0)
 =[M^0:S_i^0].
\]
Thus \eqref{c2:eq:cleft-iso-monoid} preserves composition
multiplicities under the correspondence
$S_i\leftrightarrow S_i^0$.
Since it also preserves and reflects indecomposability, property
$\X$ for $A$ implies property $\X$ for $A^0$.
\end{proof}

\section{Finite-dimensional algebras over fields}\label{sec:field-cases}

We treat finite and infinite base fields separately. Over a finite
field, representation embeddings provide the required counting
argument. Over an infinite field, we first consider algebras without
simple self-extensions and then pass to graded modules.

\subsection{Finite fields}\label{sec:finite-fields}

\begin{proposition}
\label{c2:prop:finite-field}
Let $k$ be a finite field and let $A$ be a finite-dimensional
$k$-algebra satisfying~$\X$. Then $A$ is representation-finite.
\end{proposition}

\begin{proof}
We may assume that $A\ne0$. Let
$S_1,\ldots,S_s$ be representatives of the isomorphism classes
of simple left $A$-modules. For a finite-dimensional algebra $C$
over a field $F$ and a positive integer $d$, let $b_C(d)$ denote
the number of isomorphism classes of indecomposable left
$C$-modules of $F$-dimension at most $d$.

If $M$ is an $A$-module with $\dim_kM\le d$, then
\[
 0\le [M:S_i]\le d\qquad(1\le i\le s).
\]
Since $\X(A)$ holds, the composition multiplicities determine
the isomorphism class of an indecomposable $A$-module. Hence
\begin{equation}\label{c2:eq:finite-field-original-count}
 b_A(d)\le(d+1)^s.
\end{equation}

By \cite[Lemma~6.1.1]{DrozdKirichenko}, there exists a finite extension $K/k$ such that $K\otimes_kA$ is a split $K$-algebra. Put
\[
 h=[K:k],\qquad B=K\otimes_kA,\qquad |K|=q.
\]

We claim that
\begin{equation}\label{c2:count-bound}
 b_B(d)\le hd\,b_A(hd)\le hd(hd+1)^s
 \qquad(d\ge1).
\end{equation}

The extension $K/k$ is separable, since finite fields are perfect;
see \cite[Corollary~5.3.9]{DrozdKirichenko}.
By \cite[Theorem~6.1.6]{DrozdKirichenko}, the multiplication map
\[
 K\otimes_kK\longrightarrow K
\]
splits as a $K$-$K$-bimodule homomorphism. For a left $B$-module $N$, it follows that $N$ is a direct summand of
$K\otimes_k N$.

Suppose now that $N$ is an indecomposable $B$-module and $\dim_KN\le d$.
Choose a decomposition of $N$ into indecomposable $A$-modules
$N \cong M_1\oplus\cdots\oplus M_u.$
By the preceding splitting and the Krull--Schmidt theorem, $N$ is isomorphic, as a $B$-module, to
an indecomposable direct summand of $K\otimes_kM_j$ for some $j$.
Moreover,
\[
 \dim_kM_j
 \le\dim_k N
 =h\dim_KN\le hd.
\]
For each indecomposable $A$-module $M$ with $\dim_kM\le hd$,
the module $K\otimes_kM$ has $K$-dimension at most $hd$, and
therefore has at most $hd$ nonzero indecomposable direct summands.
There are at most $b_A(hd)$ possible isomorphism classes of such
$M$. Thus
$b_B(d)\le hd\,b_A(hd),$
and \eqref{c2:eq:finite-field-original-count} proves
\eqref{c2:count-bound}.

Suppose, towards a contradiction, that $A$ is
representation-infinite. Since $K/k$ is a finite separable extension, it is
MacLane-separable. Hence
\cite[Theorem~3.3]{JensenLenzing1982} gives
\[
A\text{ is representation-finite}
\quad\Longleftrightarrow\quad
B\text{ is representation-finite}.
\]
Consequently, $B$ is representation-infinite.
Since $B$ is a $K$-split algebra, \cite[Theorem~6(i) and the following remark]{Bongartz} gives a bimodule
${}_BZ_{K[X]}$ such that
$Z_{K[X]}\cong K[X]^r$
for some positive integer $r$, and the functor
\[
 Z\otimes_{K[X]}-:
 K[X]\text{-}\mathrm{mod}_0
 \longrightarrow B\text{-}\mathrm{mod}
\]
is a representation embedding, that is, $Z\otimes_{K[X]}-$ is an exact functor preserving indecomposable modules and reflecting isomorphisms. Here $K[X]\text{-}\mathrm{mod}_0$ is the category of
finite-dimensional $K[X]$-modules.

Let $p$ be a prime number, and let $\mathcal P_p$ be the set of
monic irreducible polynomials of degree $p$ in $K[X]$.
Choose an extension $K_p/K$ of degree $p$; thus $|K_p|=q^p$.
For every $\alpha\in K_p\setminus K$, the tower formula gives
$[K(\alpha):K]\mid [K_p:K]=p.$
Since $\alpha\notin K$, its minimal polynomial over $K$ has
degree $p$.
Conversely, every polynomial in $\mathcal P_p$ has exactly $p$
distinct roots in $K_p$: its splitting field over $K$ is the
finite field with $q^p$ elements, and finite fields are perfect.
Distinct monic irreducible polynomials have disjoint sets of
roots. Hence these root sets partition $K_p\setminus K$, and
therefore
\[
 p|\mathcal P_p|=q^p-q,\qquad
 |\mathcal P_p|=\frac{q^p-q}{p}.
\]

For each $f\in\mathcal P_p$, put
\[
 U_f=K[X]/(f),\qquad
 N_f=Z\otimes_{K[X]}U_f.
\]
Each $U_f$ is a simple $K[X]$-module of $K$-dimension $p$.
Moreover,
$\ann_{K[X]}(U_f)=(f),$
so the modules $U_f$, for $f\in\mathcal P_p$, are pairwise
nonisomorphic.
The representation embedding constructed above preserves
indecomposability and reflects isomorphisms. Consequently, the
modules $N_f$, for $f\in\mathcal P_p$, are pairwise nonisomorphic
indecomposable left $B$-modules.
Since $Z_{K[X]}\cong K[X]^r$, there is an isomorphism of
$K$-vector spaces
\[
 N_f
 =Z\otimes_{K[X]}U_f
 \cong U_f^{\oplus r}.
\]
Thus $\dim_KN_f=rp$ for every $f\in\mathcal P_p$.
By the definition of $b_B$ and \eqref{c2:count-bound}, we obtain
\[
 \frac{q^p-q}{p}
 =|\mathcal P_p|
 \le b_B(rp)
 \le hrp(hrp+1)^s,
\]
and hence
\[
 1\le
 \frac{hrp^2(hrp+1)^s+q}{q^p}
\]
for every prime number $p$.
However, $q\ge2$ and $h,r,s$ are fixed, so
\[
 \lim_{n\to\infty}
 \frac{hrn^2(hrn+1)^s+q}{q^n}=0.
\]
Since the prime numbers are unbounded, this contradicts the
preceding inequality. Hence $A$ is representation-finite.
\end{proof}

\subsection{Infinite fields}\label{sec:infinite-fields}
\label{c2:sec:brickification}\label{c2:sec:matrix}

In Lemma~\ref{c2:lem:rational-brick-family} and Proposition~\ref{c2:thm:loopless}, $k$ is an infinite field and $A$ is a
finite-dimensional $k$-algebra. We consider the exact sequence
	\[
	0\longrightarrow k[X]/(X)\xrightarrow{\cdot X}k[X]/(X^2)
	\longrightarrow k[X]/(X)\longrightarrow0 \qquad (\natural)
	\]
in $k[X]\modu$, where the second homomorphism is the natural projection.

\begin{lemma}\label{c2:lem:rational-brick-family}
Suppose that $A$ satisfies $\X$. Let $F$ be an
$A\otimes_k k[X]$-module which is free of finite rank over $k[X]$.
Assume that $F/XF$ is a brick as an $A$-module and that the
$A$-modules $F/(X-\lambda)F$, for all $\lambda\in k$, have the same
composition factors. Then the sequence
\begin{equation}\label{c2:eq:family-first-order-extension}
 0\longrightarrow F/XF\xrightarrow{j}F/X^2F
 \xrightarrow{\pi}F/XF\longrightarrow0
\end{equation}
is split exact as a sequence of $A$-modules. Here $j$ is induced
by multiplication by $X$, and $\pi$ is the natural projection.
\end{lemma}

\begin{proof}
The sequence \eqref{c2:eq:family-first-order-extension} is isomorphic to $F\otimes_{k[X]} (\natural)$. Since $F$ is free over $k[X]$, the sequence
\eqref{c2:eq:family-first-order-extension} is exact.

Since $k$ is infinite and $A$ satisfies $\X$, there is an infinite subset $\Lambda$ of $k$ and a finite-dimensional $A$-module $N$ such
that
\begin{equation*}
 u_\lambda:N\to F/(X-\lambda)F
\end{equation*}
is an $A$-module isomorphism for all $\lambda\in \Lambda$.

Let $E=\Ext_A^1(N,F)$. A projective resolution of $N$ with finitely generated terms shows that $E$ is finitely generated
over $k[X]$.
Consider the finitely generated $k[X]$-module $E\oplus F$ and the family $\{X-\lambda\}_{\lambda\in \Lambda}\subseteq k[X]$. Since $\Lambda$ is an infinite set, there is an element $\mu\in \Lambda$ such that $m_{X-\mu}:E\oplus F\to E\oplus F$ is injective by Lemma~\ref{c2:lem:finite-exceptional-elements}, where $m_{X-\mu}$ is multiplication by $X-\mu$.
Hence
\[
 0\longrightarrow F\xrightarrow{m_{X-\mu}}F
 \xrightarrow{q_{X-\mu}}F/(X-\mu)F\longrightarrow0
\]
is a short exact sequence of left $A$-modules.

Applying $\Hom_A(N,-)$ therefore gives the exact sequence
\[
 \Hom_A(N,F)\xrightarrow{\Hom_A(N,q_{X-\mu})}
 \Hom_A\bigl(N,F/(X-\mu)F\bigr)
 \longrightarrow E\xrightarrow{\Ext_A^1(N,m_{X-\mu})}E.
\]
Note that the map $\Ext_A^1(N,m_{X-\mu})$ is exactly
multiplication by $X-\mu$ on $E$. This implies that $\Ext_A^1(N,m_{X-\mu})$ is injective and hence $\Hom_A(N,q_{X-\mu})$ is surjective.
There is an $A$-module homomorphism
\begin{equation*}
 f:N\longrightarrow F,
\end{equation*}
such that $fq_{X-\mu}=u_\mu$.
In particular, $f\neq0$.

Define
\[
 \widetilde f:N\otimes_k k[X]\longrightarrow F,
 \qquad (x\otimes p(X))\widetilde f=p(X)((x)f).
\]
Clearly, $\widetilde f$ is
an $A\otimes_k k[X]$-module homomorphism. The natural $A$-module isomorphism
\[
 \frac{N\otimes_k k[X]}{(X-\mu)(N\otimes_k k[X])}
 \cong N
\]
is induced by $x\otimes p(X)\mapsto p(\mu)x$; its inverse
sends $x$ to $x\otimes1$ in the quotient module. In that quotient,
$x\otimes p(X)=p(\mu)(x\otimes1)$.
Under this isomorphism, the map induced by $\widetilde f$ on the
quotients by $X-\mu$ is $fq_{X-\mu}=u_\mu$.
If $r$ is the $k[X]$-rank of $F$, a free basis gives
$F/(X-\mu)F\cong k^r$ as $k$-vector spaces. Hence
$\dim_kN=r$, and both $N\otimes_k k[X]$ and $F$
are free $k[X]$-modules of rank $r$. Let $K:=k(X)$.
Applying Lemma~\ref{aux} to $\widetilde f$, we have
\begin{equation}\label{c2:eq:generic-fixed-module}
 N\otimes_kK\cong N\otimes_k k[X]\otimes_{k[X]}K\cong F\otimes_{k[X]}K
\end{equation}
as $A\otimes_kK$-modules.

We next prove that $F\otimes_{k[X]}K$ is a brick as an
$A\otimes_kK$-module. Let
\[
\iota:F\longrightarrow F\otimes_{k[X]}K,
\qquad
(x)\iota=x\otimes1,
\]
be the natural $A\otimes_k k[X]$-module homomorphism.
Since $F$ is free over $k[X]$, the map $\iota$
is injective.

Take
\[
0\ne g\in
\operatorname{End}_{A\otimes_kK}
\bigl(F\otimes_{k[X]}K\bigr).
\]
Choose a finite set $x_1,\ldots,x_r$ of $k[X]$-module
generators of $F$.
Since $K$ is the fraction field of $k[X]$, for each $i$
there exists a nonzero polynomial $q_i(X)\in k[X]$ such that
$q_i(X)\bigl((x_i)\iota g\bigr)\in(F)\iota.$
Put $
p(X)=\prod_{i=1}^r q_i(X).
$
Then $p(X)\ne0$. Since $\iota g$ is a $k[X]$-module homomorphism, we obtain
$p(X)\bigl((F)\iota g\bigr)\subseteq(F)\iota.$

For each $x\in F$, the above inclusion and the injectivity
of $\iota$ give a unique element $(x)w\in F$ satisfying
\[
\bigl((x)w\bigr)\iota
=
p(X)\bigl((x)\iota g\bigr).
\]
This defines an $A\otimes_k k[X]$-module homomorphism $w:F\to F$.
By construction,
$w\iota=p(X)(\iota g).$
Since $(F)\iota$ spans $F\otimes_{k[X]}K$ over $K$ and
$g\ne0$, we have $\iota g\ne0$.
As $p(X)$ is invertible in $K$, it follows that $w\ne0$.

Apply Lemma~\ref{c2:lem:normalization-over-domain} with $R=k[X]$, $a=X$, and $V=F$.
There is an integer $m\ge0$
and an $A\otimes_k k[X]$-module homomorphism
\[
h:F\longrightarrow F
\]
such that
\[
w=X^m h,
\qquad
(F)h\nsubseteq XF.
\]
Thus $h$ induces a nonzero $A$-module endomorphism
\[
F/XF\longrightarrow F/XF.
\]
Since $F/XF$ is a brick as an $A$-module, this induced endomorphism is an
$A$-module isomorphism. By Lemma~\ref{aux}, we have an $A\otimes_kK$-module isomorphism
\[
h\otimes_{k[X]}1_K:
F\otimes_{k[X]}K\longrightarrow F\otimes_{k[X]}K.
\]

Note that
$
g=\frac{X^m}{p(X)}(h\otimes_{k[X]}1_K),
$
and the scalar $X^m/p(X)$ is invertible in $K$. Thus $g$ is an
$A\otimes_kK$-module isomorphism.
Finally, $F/XF\ne0$ implies that the free $k[X]$-module $F$
has positive rank, and hence $F\otimes_{k[X]}K\ne0$.
Therefore $F\otimes_{k[X]}K$ is a brick.

By~\eqref{c2:eq:generic-fixed-module}, $N\otimes_kK$ is an
indecomposable $A\otimes_k K$-module. Hence $N$ is an indecomposable $A$-module. Clearly, $F/XF$ is also an indecomposable $A$-module.
Since $N$ and $F/XF$ have the same composition factors, $N\cong F/XF$ as $A$-modules by property $\X$. We denote this isomorphism $F/XF\to N$ by $\alpha$.

Apply Lemma~\ref{c2:lem:normalization-over-domain} to the nonzero map $\omega = \alpha f : F/XF \to F$.
We obtain an $A$-module homomorphism $v : F/XF \to F$ whose image is not contained in $XF$.
Let $q:F\to F/XF$ and $\rho:F\to F/X^2F$ be the natural
$A$-module projections. Then $\rho\pi=q$, and $vq$ is a nonzero
$A$-module endomorphism of $F/XF$. Since $F/XF$ is a brick as an $A$-module, $vq$ is an automorphism of $F/XF$.

Set
\[
 s=(vq)^{-1}v\rho:F/XF\longrightarrow F/X^2F.
\]
This is an $A$-module homomorphism, and
\[
 s\pi=(vq)^{-1}v\rho\pi=(vq)^{-1}vq=1_{F/XF}.
\]
This proves that
\eqref{c2:eq:family-first-order-extension} splits.
\end{proof}

\begin{proposition}\label{c2:thm:loopless}
Suppose that $A$ satisfies $\X$ and that
$\Ext_A^1(S,S)=0$ for every simple $A$-module $S$.
Then every indecomposable $A$-module is a brick, and $A$ has finite
representation type.
\end{proposition}

\begin{proof}
	Suppose that some indecomposable $A$-module is not a brick. By
	Lemma~\ref{c2:lem:shortest-nonbrick}, there exist a brick $B$
	and a nonzero class
	$[\xi]\in\Ext_A^1(B,B)$
	whose pullback along the inclusion of any proper submodule of
	$B$ is zero. The module $B$ is not simple, since otherwise
	the hypothesis would give $\Ext_A^1(B,B)=0$.
	
	Choose an epimorphism $p:B\longrightarrow S$ with $S$ simple,
	and let $u:U\longrightarrow B$ be its kernel. Thus
	\[
	0\longrightarrow U\xrightarrow{u}B\xrightarrow{p}S
	\longrightarrow0
	\]
	is exact and $0\ne U\subsetneq B$. The associated long exact
	sequence contains
	\[
	\Ext_A^1(S,B)\xrightarrow{p^*}\Ext_A^1(B,B)
	\xrightarrow{u^*}\Ext_A^1(U,B).
	\]
	Since $([\xi])u^*=0$, there exists
	$[\eta]\in\Ext_A^1(S,B)$ such that
	$([\eta])p^*=[\xi].$
	Choose a short exact sequence representing $[\eta]$:
	\begin{equation*}
		0\longrightarrow B\xrightarrow{i}E\xrightarrow{q}S
		\longrightarrow0.
	\end{equation*}
	
	Applying $\Hom_A(-,S)$ to this sequence gives
	\[
	\Hom_A(E,S)\xrightarrow{\Hom_A(i,S)}\Hom_A(B,S)
	\longrightarrow\Ext_A^1(S,S)=0.
	\]
	Consequently, there exists an $A$-module homomorphism
	$r:E\longrightarrow S$ such that
	$ir=p.$
	
	Let $g:=(r,q):E\longrightarrow S\oplus S$, $j_1:=(1_S,0):S\longrightarrow S\oplus S$ and
	$\pi_2:=\left(\begin{smallmatrix}0\\1_S\end{smallmatrix}\right):S\oplus S\longrightarrow S$ be $A$-module homomorphisms.
	There is a commutative diagram with exact rows
	\[
	\begin{tikzcd}[column sep=large, row sep=large]
    0 \arrow{r}
      & B \arrow{r}{i} \arrow{d}{p}
      & E \arrow{r}{q} \arrow{d}{g}
      & S \arrow{r} \arrow{d}{1}
      & 0 \\
    0 \arrow{r}
      & S \arrow{r}{j_1}
      & S\oplus S \arrow{r}{\pi_2}
      & S \arrow{r}
      & 0 .
    \end{tikzcd}
	\]
	Since $p$ is an epimorphism, the snake lemma shows that $g$
	is an epimorphism.
	
	Put $T=k[X]$. For an $A$-module $Y$ and an $A$-module homomorphism $f$, write
	\[
	Y_T=Y\otimes_kT,
	\qquad
	f_T=f\otimes_k1_T.
	\]
	Define
	\[
	\varphi=q_T-X(r_T):E_T\longrightarrow S_T,
	\]
	and let $v:F\longrightarrow E_T$ be its kernel.
	The morphism $\varphi$ is the composite of the epimorphism
	$g_T:E_T\longrightarrow S_T\oplus S_T$ with the split
	epimorphism
	\[
	\left(\begin{smallmatrix}-X\\1_{S_T}\end{smallmatrix}\right):
	S_T\oplus S_T\longrightarrow S_T.
	\]
	Hence
	\begin{equation}
		\label{c2:eq:kernel-family-sequence}
		0\longrightarrow F\xrightarrow{v}E_T
		\xrightarrow{\varphi}S_T\longrightarrow0
	\end{equation}
	is an exact sequence of $A\otimes_kT$-modules.
	
	As a sequence of $T$-modules,
	\eqref{c2:eq:kernel-family-sequence} splits, because $S_T$
	is free over $T$. Thus $F$ is a finitely generated
	projective $T$-module. Since $T$ is a principal ideal
	domain, $F$ is a free $T$-module of finite rank. Moreover, every base
	change of \eqref{c2:eq:kernel-family-sequence} remains exact.
	
	For $\lambda\in k$, put
	$F_\lambda=F/(X-\lambda)F.$
	For every $A$-module $Y$, denote by
	\[
	\epsilon_{Y,\lambda}:
	Y_T/(X-\lambda)Y_T\xrightarrow{\sim}Y
	\]
	the isomorphism induced by evaluation at $X=\lambda$.
	Let
	\[
	v_\lambda:F_\lambda\longrightarrow E
	\]
	be the composite of the morphism induced by $v$ with
	$\epsilon_{E,\lambda}$. Specializing
	\eqref{c2:eq:kernel-family-sequence} gives an exact sequence in $A\modu$
	\[
	0\longrightarrow F_\lambda\xrightarrow{v_\lambda}E
	\xrightarrow{q-\lambda r}S\longrightarrow0.
	\]
	It follows that $F_\lambda$ and $B$ have the same composition factors as $A$-modules for every $\lambda\in k$.
	
	Put
	\[
	V=F/XF,\qquad H=F/X^2F.
	\]
	At $\lambda=0$, the specialized sequence is
	\[
	0\longrightarrow V\xrightarrow{v_0}E
	\xrightarrow{q}S\longrightarrow0.
	\]
	Both $i:B\longrightarrow E$ and
	$v_0:V\longrightarrow E$ are kernels of $q$.
	The universal property of the kernel therefore gives
	a unique isomorphism
	\[
	\alpha:V\xrightarrow{\sim}B,
	\qquad \alpha i=v_0.
	\]
    In particular, $V$ is a brick.
	
	Since $F$ is free over $T$, tensoring the exact sequence
	\[
	0\longrightarrow T/(X)\xrightarrow{\cdot X}T/(X^2)
	\longrightarrow T/(X)\longrightarrow0
	\]
	with $F$ gives the natural short exact sequence of
	$A$-modules
	\begin{equation*}
		\xi_F:\quad
		0\longrightarrow V\xrightarrow{\iota}H
		\xrightarrow{\pi}V\longrightarrow0.
	\end{equation*}
	Lemma~\ref{c2:lem:rational-brick-family} gives
	\begin{equation}
		\label{c2:eq:family-extension-zero}
		[\xi_F]=0
		\qquad\text{in }\Ext_A^1(V,V).
	\end{equation}
	
	We compare $[\xi_F]$ with $[\xi]$ using the naturality of Yoneda extensions. Put $D=T/(X^2)$, and denote the image of $X$
	in $D$ by $\bar X$. For every $A$-module $Y$, write
	$Y_D=Y\otimes_kD$. There is a natural split exact sequence
	of $A$-modules
	\[
	0\longrightarrow Y\xrightarrow{a_Y}Y_D
	\xrightarrow{b_Y}Y\longrightarrow0,
	\]
	where $a_Y$ is induced by the homomorphism of $k$-vector spaces
	$k\longrightarrow D$ sending $1$ to $\bar X$, and $b_Y$
	is induced by the augmentation $D\longrightarrow k$.
	There are
	\[
	c_Y:Y\longrightarrow Y_D,\qquad d_Y:Y_D\longrightarrow Y
	\]
	such that
	\[
	c_Yb_Y=1_Y,\qquad a_Y d_Y=1_Y,\qquad d_Ya_Y+b_Yc_Y=1_{Y_D},\qquad c_Y d_Y=0.
	\]
	All these morphisms are natural in $Y$.
	
	Let $m_Y=b_Y a_Y:Y_D\longrightarrow Y_D$. Then
	$m_Y d_Y=b_Y$.
	Reduction of $v$ modulo $X^2$ induces a monomorphism
	\[
	v_2:H\longrightarrow E_D.
	\]
	Compatibility with reduction and multiplication by $X$
	gives
	\begin{equation}
		\label{c2:eq:base-change-compatibility}
    v_2 b_E=\pi v_0,
		\qquad\iota v_2=v_0 a_E.
	\end{equation}
	Moreover, reduction of $v \varphi=0$ gives
	$v_2(q_D-r_D m_S)=0,$
	where $q_D=q\otimes_k1_D$ and $r_D=r\otimes_k1_D$.
	
	Define
	\[
	h=v_2 d_E:H\longrightarrow E.
	\]
	By \eqref{c2:eq:base-change-compatibility},
	\[
	\iota h
	=v_0 a_E d_E
	=v_0
	=\alpha i.
	\]
	Naturality of the morphisms $b_Y$ and $d_Y$ gives
	\[
    \begin{aligned}
		h q & = v_2 q_D d_S = v_2 r_D m_S d_S = v_2 r_D b_S\\
		    &=v_2 b_E r =\pi v_0 r =\pi \alpha i r =\pi \alpha p.
	\end{aligned}
    \]
	Consequently, there is a commutative diagram of short
	exact sequences
	\[
	\begin{tikzcd}[column sep=large, row sep=large]
    0 \arrow{r}
      & V \arrow{r}{\iota} \arrow{d}{\alpha}
      & H \arrow{r}{\pi} \arrow{d}{h}
      & V \arrow{r} \arrow{d}{\alpha p}
      & 0 \\
    0 \arrow{r}
      & B \arrow{r}{i}
      & E \arrow{r}{q}
      & S \arrow{r}
      & 0 .
    \end{tikzcd}
	\]
	
	Naturality of Yoneda extensions now yields
	\begin{equation}
		\label{c2:eq:yoneda-comparison}
		([\xi_F])\alpha_*
		=([\eta])(\alpha p)^*
		=(([\eta]) p^*)\alpha^*
		=([\xi])\alpha^*
		\qquad\text{in }\Ext_A^1(V,B).
	\end{equation}
	Here $\alpha_*$ denotes pushout along $\alpha$ in the second
	variable, and $\alpha^*$ denotes pullback along $\alpha$ in
	the first variable.
	
	By \eqref{c2:eq:family-extension-zero}, the left-hand side
	of \eqref{c2:eq:yoneda-comparison} is zero. Since $\alpha$
	is an isomorphism, $[\xi]=0$, contradicting its choice.
	
	Every indecomposable finitely generated left $A$-module
	is consequently a brick. By
	\cite[Corollary~4.6]{Skowronski}, $A$ has finite representation type.
\end{proof}

We next pass to graded modules, using the notation of
Subsection~\ref{sec:prelim-graded}. The first proposition gives a
finiteness criterion for modules with support in a fixed finite set.
The following proposition supplies the required uniform support bound.
In these propositions,
$\Lambda$ is a finite-dimensional basic $\mathbb Z^r$-graded algebra,
and $f_1,\ldots,f_s$ are primitive orthogonal idempotents of degree zero.

\begin{proposition}
\label{c2:prop:finite-support-criterion}
Let $k$ be an infinite field, and let $\Lambda$ be a
finite-dimensional basic $\mathbb Z^r$-graded $k$-algebra
satisfying $\X$.
Let $f_1,\ldots,f_s\in\Lambda_0$ be a complete set of
pairwise orthogonal primitive idempotents, and put $J=J(\Lambda)$.
Assume that
\begin{equation}\label{c2:eq:zero-degree-self-extension-condition}
 f_i(J/J^2)_0f_i=0
 \qquad(1\le i\le s).
\end{equation}
Then, for every finite nonempty subset $I\subseteq\mathbb Z^r$,
the category $\Lambda\modu_{\mathrm{gr}}^I$
has only finitely many isomorphism classes of indecomposable
objects. Moreover, every indecomposable graded $\Lambda$-module
$M$ has a division endomorphism ring
$\operatorname{End}_{\mathrm{gr}\Lambda}(M)$.

These two conclusions hold, in particular, if every
$f_i\Lambda_0f_i$ is a division ring.

If, in addition, there exists a finite nonempty subset
$I_0\subseteq\mathbb Z^r$ such that every indecomposable
graded $\Lambda$-module has a shift supported in $I_0$,
then $\Lambda$ is representation-finite.
\end{proposition}

\begin{proof}
Fix a finite nonempty subset $I\subseteq\mathbb Z^r$, and put
$\mathcal C=\Lambda\modu_{\mathrm{gr}}^I$.
By \cite[Theorem~2.1]{Farnsteiner2009}, there is a
finite-dimensional $k$-algebra $\Gamma$ and an equivalence
\[
 \Phi:\mathcal C\xrightarrow{\sim}\Gamma\modu.
\]
By Lemma~\ref{c2:lem:graded-composition-factors}, any two
indecomposable objects of $\mathcal C$ with the same composition
multiplicities are isomorphic. Since $\Phi$ is an equivalence, $\Gamma$ satisfies $\X$.

By Lemma~\ref{c2:lem:graded-simple-self-extensions} and
\eqref{c2:eq:zero-degree-self-extension-condition}, every
simple object of $\mathcal C$ has no nontrivial self-extension.
The equivalence $\Phi$ is exact, so
\[
 \operatorname{Ext}^1_\Gamma(T,T)=0
 \qquad\text{for every simple $\Gamma$-module $T$}.
\]
Proposition~\ref{c2:thm:loopless} now implies that $\Gamma$
has finite representation type and that every indecomposable
left $\Gamma$-module is a brick. Transferring these conclusions
through $\Phi$ proves that $\mathcal C$ has finitely many
isomorphism classes of indecomposable objects and that each of its
indecomposable objects has a division endomorphism ring.

For an arbitrary indecomposable graded $\Lambda$-module $M$,
apply the preceding argument with $I=\operatorname{supp}(M)$.
This proves the assertion about
$\operatorname{End}_{\mathrm{gr}\Lambda}(M)$.
The assertion concerning division corners follows from
Lemma~\ref{c2:lem:graded-simple-self-extensions}.

It remains to prove the assertion under the additional
uniform support hypothesis. Fix the finite set $I_0$ and
choose representatives $N_1,\ldots,N_t$ of the isomorphism classes of indecomposable objects of $\Lambda\modu_{\mathrm{gr}}^{I_0}$.
Indecomposability in this category agrees with
indecomposability in the full graded module category,
since graded direct summands remain supported in $I_0$.
Let $F$ denote the forgetful functor. By
\cite[Corollary~1.3]{Farnsteiner2009}, each $F(N_j)$ is
indecomposable. Conversely, any ungraded indecomposable
module admitting a grading is indecomposable with that
grading, and the hypothesis puts some shift of it in
$\Lambda\modu_{\mathrm{gr}}^{I_0}$.
Since shifts do not change the underlying module, every
gradable ungraded indecomposable $\Lambda$-module is
isomorphic to one of $F(N_1),\ldots,F(N_t)$.

We now use the $\mathbb Z^r$-graded form of the implication
\textup{(ii)}$\Rightarrow$\textup{(i)} in
\cite[Theorem~4.5]{GordonGreen1982}.
For completeness, its proof in the present setting is as follows.
Let $C$ be any block of the underlying ungraded algebra $\Lambda$,
and choose an indecomposable projective $C$-module $P$.
As an ungraded $\Lambda$-module, $P$ is projective and is
gradable by \cite[Proposition~1.4(2)]{Farnsteiner2009}.
Let $\Theta$ be the component containing $P$ in the Auslander--Reiten quiver of $\Lambda$.
By \cite[Corollary~2.4]{Farnsteiner2009}, every module in
$\Theta$ is gradable. Hence $\Theta$ has finitely many vertices,
and the composition lengths of its modules are bounded.
The ungraded module category of $\Lambda$ is the product of
the module categories of its blocks, so $\Theta$ is also a
component of the Auslander--Reiten quiver of $C$.
By \cite[Theorem~VI.1.4, p.~193]{ARS}, $C$
is representation-finite. Since $\Lambda$ has finitely
many blocks, $\Lambda$ is representation-finite as well.
\end{proof}

The next proposition bounds the support of an indecomposable graded
module in each coordinate direction. The bounds depend only on the
Loewy lengths of the primitive corners.

\begin{proposition}
\label{c2:prop:bounded-decomposition}
Let $k$ be an infinite field. In the above notation, assume that
$r=s$ and that $\Lambda$ satisfies $\X$, and let
$\varepsilon_1,\ldots,\varepsilon_s$ be the standard basis of
$\mathbb Z^s$. Put $J=J(\Lambda)$.
Suppose that $f_i(J/J^2)f_i$ is concentrated
in degree $\varepsilon_i$ for every $i$, and that
$f_i(J/J^2)f_j$ is concentrated in degree zero whenever $i\ne j$.
Then every finite-dimensional indecomposable graded
$\Lambda$-module $M$ satisfies
\[
 \max_{v\in\operatorname{supp}(M)}v_i
 -
 \min_{v\in\operatorname{supp}(M)}v_i
 \le \operatorname{LL}(f_i\Lambda f_i)-1
 \qquad(1\le i\le s).
\]
\end{proposition}

\begin{proof}
Fix a nonzero indecomposable graded $\Lambda$-module $M$. Put
\[
 B_i=f_i\Lambda f_i,\qquad
 J_i=J(B_i)=f_iJf_i,\qquad
 \ell_i=\operatorname{LL}(B_i),
\]
and regard
\[
 B=\prod_{i=1}^sB_i,\qquad
 H=\bigoplus_{i\ne j}f_i\Lambda f_j
\]
as subspaces of $\Lambda$, with $B$ a unital subalgebra.

We will use two consequences of the hypotheses throughout the proof.
First, the underlying ungraded $\Lambda$-module of $M$ is indecomposable by
\cite[Corollary~1.3]{Farnsteiner2009}.
Second, since $f_i(J/J^2)f_i$ is concentrated in the nonzero
degree $\varepsilon_i$, one has
\[
 f_i(J/J^2)_0f_i=0\qquad(1\le i\le s).
\]
Proposition~\ref{c2:prop:finite-support-criterion} therefore implies
that $\operatorname{End}_{\mathrm{gr}\Lambda}(M)$ is a division ring.
In particular, $M$ has no nonzero degree-zero endomorphism whose
square is zero.

\smallskip
\noindent\textbf{Claim 1.}
The radical $J$ is homogeneous, $\Lambda/J$ is concentrated
in degree zero, and each $B_i$ is a local Nakayama algebra.
Moreover, with $J_i^0=B_i$, one has
\begin{equation}
\label{c2:eq:abstract-corner-grading}
 J_i^q
 =
 \bigoplus_{p=q}^{\ell_i-1}(B_i)_{p\varepsilon_i}
 \quad(0\le q\le\ell_i),
 \qquad
 (B_i)_0\cong C_i=B_i/J_i.
\end{equation}
The subspace $H$ is concentrated in degree zero, and
\begin{equation}
\label{c2:eq:abstract-diagonal-decomposition}
 \Lambda=B\oplus H,\qquad
 J=J(B)\oplus H,\qquad
 J(B)H=HJ(B)=0.
\end{equation}
The direct sums in \eqref{c2:eq:abstract-diagonal-decomposition}
are direct sums of graded vector spaces.

\smallskip
\noindent\emph{Proof of Claim 1.}
Since $\Lambda$ is basic, $\Lambda/J$ is a product of division
algebras and hence has no nonzero nilpotent elements.
Lemma~\ref{c2:prop:graded-radical-reduced-quotient} shows that
$J$ is homogeneous and $\Lambda/J$ is concentrated in degree zero.
As the $f_i$ have degree zero, all the subspaces $B_i$, $J_i$,
$B$, and $H$ are homogeneous. Moreover, basicness gives
$f_i\Lambda f_j=f_iJf_j$ for $i\ne j$. Thus
\[
 \Lambda=B\oplus H,\qquad
 J=\bigoplus_iJ_i\oplus H=J(B)\oplus H.
\]
In particular, $J(B)\subseteq J$.
By Lemma~\ref{c2:lem:no-return},
\begin{equation}\label{iajai}
 f_i\Lambda f_j\Lambda f_i=0\qquad(i\ne j),
\end{equation}
and there is an algebra epimorphism $\Lambda\to B_i$. Thus $B_i$ satisfies $\X$ and is a local Nakayama
algebra by Lemma~\ref{c2:lem:equal-top-socle}(2).
Lemma~\ref{2.3} also gives
\begin{equation}
\label{c2:eq:abstract-corner-annihilation}
 J_i(f_i\Lambda f_j)=0=(f_i\Lambda f_j)J_j
 \qquad(i\ne j).
\end{equation}
By (\ref{iajai}), we have
\[
 f_iJ^2f_i
 =
 J_i^2+\sum_{j\ne i}f_iJf_jJf_i
 =
 J_i^2.
\]
Consequently, the map
\[
 J_i/J_i^2\longrightarrow f_i(J/J^2)f_i,\qquad
 x+J_i^2\longmapsto x+J^2
\]
is an isomorphism of graded bimodules. In particular,
$J_i/J_i^2$ is concentrated in degree $\varepsilon_i$.

The quotient $B_i/J_i$ is concentrated in degree zero.
For $q\ge1$, every element of $J_i^q/J_i^{q+1}$ is a sum of
classes of products $x_1\cdots x_q$, with $x_\nu\in J_i$.
Choose $u_\nu\in(J_i)_{\varepsilon_i}$ such that
$x_\nu-u_\nu\in J_i^2$. Expanding the product gives the equality
in the quotient
\[
 x_1\cdots x_q+J_i^{q+1}
 =
 u_1\cdots u_q+J_i^{q+1}.
\]
Indeed, every other term contains one factor from $J_i^2$
and $q-1$ factors from $J_i$.
Thus $J_i^q/J_i^{q+1}$ is concentrated in degree
$q\varepsilon_i$.

In the finite filtration
$B_i\supseteq J_i\supseteq\cdots\supseteq J_i^{\ell_i}=0$,
the only possibly nonzero quotient in degree $p\varepsilon_i$
is $J_i^p/J_i^{p+1}$. For $q=0,\ldots,p-1$, the degree-$p\varepsilon_i$ component
of $J_i^q/J_i^{q+1}$ is zero. Hence
$(J_i^q)_{p\varepsilon_i} = (J_i^{q+1})_{p\varepsilon_i}.$
Successively applying these equalities, with $J_i^0=B_i$, gives
\[
 (J_i^q)_{p\varepsilon_i}
 =
 (B_i)_{p\varepsilon_i}
 \qquad(0\le q\le p).
\]

Similarly, for $q=p+1,\ldots,\ell_i-1$, the
degree-$p\varepsilon_i$ component of $J_i^q/J_i^{q+1}$
is zero. Since $J_i^{\ell_i}=0$, we obtain
\[
 (J_i^{p+1})_{p\varepsilon_i}
 =
 \cdots
 =
 (J_i^{\ell_i})_{p\varepsilon_i}
 =
 0.
\]
Every degree outside $\{0,\varepsilon_i,\ldots,
(\ell_i-1)\varepsilon_i\}$ vanishes throughout the filtration.
This proves \eqref{c2:eq:abstract-corner-grading}.
In particular, the quotient map restricts to an algebra
isomorphism $(B_i)_0\to C_i$.

The identities \eqref{c2:eq:abstract-corner-annihilation} imply
$J(B)H=HJ(B)=0$. Since
$J(B)=\bigoplus_iJ_i$ is an ideal of $B$ and $\Lambda=B\oplus H$,
it is a two-sided ideal of $\Lambda$; it is homogeneous by the
preceding observations.
Form the graded quotient and its radical
\[
 \overline\Lambda=\Lambda/J(B),\qquad
 N=J(\overline\Lambda)=J/J(B).
\]
One has graded isomorphisms
\[
 N/N^2
 \cong J/(J^2+J(B))
 \cong
 (J/J^2)\big/\bigl((J^2+J(B))/J^2\bigr).
\]
The image of $J(B)$ in $J/J^2$ is exactly
\[
 (J^2+J(B))/J^2
 =
 \bigoplus_i f_i(J/J^2)f_i.
\]
Indeed, the image of $J_i=f_iJf_i$ is
$f_i(J/J^2)f_i$, and multiplication by the orthogonal
idempotents separates these diagonal components.
The remaining components are $f_i(J/J^2)f_j$ with $i\ne j$,
all of degree zero by hypothesis. Hence $N/N^2$ is
concentrated in degree zero.

The same radical-filtration argument shows that $N^q/N^{q+1}$ is concentrated in degree zero
for every $q \ge 1$. Since $\overline\Lambda/N\cong\Lambda/J$ is also concentrated in degree zero and $N$ is nilpotent,
$\overline{\Lambda}$ is concentrated in degree zero.
For $g\ne0$, this implies $\Lambda_g\subseteq J(B)$.
Since $H\cap J(B)=0$ and $H$ is homogeneous, $H\subseteq\Lambda_0$.
Together with the decompositions established at the beginning,
this proves Claim~1.

\smallskip
\noindent\textbf{Claim 2.}
For each $i$, fix a decomposition of $f_iM$ into indecomposable
graded $B_i$-modules. Combining these decompositions gives a
graded $B$-module decomposition of $M$.
Every summand $U$ belonging to $f_iM$ has,
for some $a\in\mathbb Z^s$ and $1\le h\le\ell_i$,
a graded isomorphism
$U\cong(B_i/J_i^h)[-a].$
Its support and radical powers are
\begin{equation}
\label{c2:eq:p55-local-support}
 \operatorname{supp}(U)
 =\{a,a+\varepsilon_i,\ldots,a+(h-1)\varepsilon_i\},
 \qquad
 J_i^pU=\bigoplus_{q=p}^{h-1}U_{a+q\varepsilon_i}
 \quad(0\le p\le h).
\end{equation}
Its top is concentrated in degree $a$, and
$\operatorname{soc}_{B_i}U=U_{a+(h-1)\varepsilon_i}$ is simple.

\smallskip
\noindent\emph{Proof of Claim 2.}
Such a graded decomposition exists by induction on dimension.
By \cite[Corollary~1.3]{Farnsteiner2009}, the underlying ungraded
$B_i$-module of every summand $U$ is indecomposable.
Since $B_i$ is a local Nakayama algebra by Claim~1,
this underlying ungraded module is isomorphic to $B_i/J_i^h$
for some $1\le h\le\ell_i$.
The ideal $J_i^h$ is homogeneous by
\eqref{c2:eq:abstract-corner-grading}, so $B_i/J_i^h$
has its quotient grading. Its underlying ungraded module is
indecomposable, and hence it is also indecomposable as a
graded module. By \cite[Theorem~1.5]{Farnsteiner2009},
there exists $a\in\mathbb Z^s$ such that
$U\cong(B_i/J_i^h)[-a]$
as graded modules. Under this isomorphism, the class
$1+J_i^h$ corresponds to a homogeneous element $u\in U_a$
satisfying
\[
 U=B_i u,\qquad \operatorname{ann}_{B_i}(u)=J_i^h.
\]

Formula \eqref{c2:eq:abstract-corner-grading} now gives
the support and radical-power descriptions in
\eqref{c2:eq:p55-local-support}.
Every displayed support component is nonzero: for
$0\le q<h$, the equality $J_i^q=J_i^{q+1}$ would imply,
by iteration and nilpotence, that $J_i^q=0$, contrary
to $q<\ell_i$.
Finally, $B_i/J_i^h$ is uniserial of length $h$.
Consequently, the top of $U$ is concentrated in degree $a$,
and its socle is the simple module
$J_i^{h-1}U=U_{a+(h-1)\varepsilon_i}$.
This proves Claim~2.

\smallskip
\noindent\textbf{Claim 3.}
Every $B$-module endomorphism $\varphi$ of $M$ satisfying
\begin{equation}
\label{c2:eq:p55-extension-condition}
 (\operatorname{soc}_B M)\varphi=0,\qquad
 (M)\varphi\subseteq J(B)M
\end{equation}
is a $\Lambda$-module homomorphism.

\smallskip
\noindent\emph{Proof of Claim 3.}
By Claim~1, we have
\[
 HM\subseteq\operatorname{soc}_B M,\qquad
 HJ(B)M=0.
\]
For $h\in H$ and $m\in M$, condition
\eqref{c2:eq:p55-extension-condition} gives
$(hm)\varphi=0=h((m)\varphi).$
Thus $\varphi$ respects the action of $H$ as well as that
of $B$. Since $\Lambda=B+H$, it is a $\Lambda$-module homomorphism. This proves Claim~3.

In particular, let $U$ and $W$ be distinct summands in the
fixed graded $B$-module decomposition of $M$, and let
$\theta:U\to W$ be a $B$-module homomorphism such that
\[
 (\operatorname{soc}_B U)\theta=0,
 \qquad (U)\theta\subseteq J(B)W.
\]
Extending $\theta$ by zero on every other summand gives a
$\Lambda$-module endomorphism of $M$ whose square is zero.
Indeed, the socle and the radical commute with these finite
direct sums, so Claim~3 applies, and the extension vanishes on
$W$, which contains its image. If $\theta$ is homogeneous,
then the extension has the same degree.

\smallskip
\noindent\textbf{Claim 4.}
If $\varphi$ is a nonzero homogeneous $\Lambda$-module endomorphism
of $M$ and $\varphi^2=0$, then
\begin{equation*}
 \deg\varphi\in
 \bigcup_{j=1}^s(\mathbb Z\setminus\{0\})\varepsilon_j.
\end{equation*}

\smallskip
\noindent\emph{Proof of Claim 4.}
By Proposition~\ref{c2:prop:finite-support-criterion}, we have $\deg\varphi\neq0$.
Put $\delta=\deg\varphi\ne0$. Suppose that $\delta$ has at
least two nonzero coordinates. Set
\[
 d=\gcd(|\delta_1|,\ldots,|\delta_s|)>0,\qquad
 \delta_0=\delta/d.
\]
The vector $\delta_0$ is primitive. Let
\[
 \pi:\mathbb Z^s\longrightarrow
 G:=\mathbb Z^s/\mathbb Z\delta_0
 \cong\mathbb Z^{s-1}
\]
be the canonical projection.
Since $\delta$
has at least two nonzero coordinates, $(\varepsilon_j)\pi\ne0$ for every $j$.

Coarsen the gradings along $\pi$ by defining
\[
 \Lambda^\pi_g=\bigoplus_{(v)\pi=g}\Lambda_v,\qquad
 M^\pi_g=\bigoplus_{(v)\pi=g}M_v
 \qquad(g\in G).
\]
These formulas give a graded algebra and a graded module
because $\pi$ is a group homomorphism. The idempotents $f_j$
still have degree zero. Since $J$ and $J^2$ are homogeneous,
the grading induced on $J/J^2$ is its coarsening along $\pi$.
The underlying module of $M^\pi$ is the indecomposable module underlying $M$. Hence
$M^\pi$ is indecomposable as a graded module.
The radical component $f_j(J/J^2)f_j$ now has
degree $(\varepsilon_j)\pi\ne0$. After identifying
$G$ with $\mathbb Z^{s-1}$, Proposition~
\ref{c2:prop:finite-support-criterion} therefore gives a
division ring
$\operatorname{End}_{\mathrm{gr}\Lambda^\pi}(M^\pi)$.
But $(\delta)\pi=0$, so $\varphi$ is a nonzero degree-zero
square-zero endomorphism of $M^\pi$, a contradiction.
This proves Claim~4.

\smallskip
\noindent\textbf{Claim 5.}
Fix $i$. There exists $a\in\mathbb Z^s$ such that
$\operatorname{supp}(U)\subseteq a+\mathbb Z\varepsilon_i$
for every nonsimple summand $U$ in the chosen decomposition
of $f_iM$ into indecomposable graded $B_i$-modules.

\smallskip
\noindent\emph{Proof of Claim 5.}
Consider two distinct nonsimple summands $U,W$ with
first and last support degrees
\[
 a,\ a+q\varepsilon_i,\qquad
 c,\ c+t\varepsilon_i,\qquad q,t>0.
\]
By Claim~2, these top and socle degrees are well-defined.
The top of $U$ and the socle of $W$ are isomorphic to the
unique simple left $B_i$-module. Choose an isomorphism
between them and form the nonzero composite
\[
 U\twoheadrightarrow U/J_iU
 \xrightarrow{\ \sim\ }\operatorname{soc}_{B_i}W
 \hookrightarrow W.
\]
It is homogeneous of degree $c+t\varepsilon_i-a$.
Because $U$ and $W$ are nonsimple uniserial modules,
this map kills $\operatorname{soc}_{B_i}U$ and its image
lies in $J_iW$. Extend it by zero on all other summands
of the graded $B$-module $M$.
The extension satisfies
\eqref{c2:eq:p55-extension-condition}, so it is a
$\Lambda$-module homomorphism by Claim~3. Clearly, it is square-zero.
Interchanging $U$ and $W$ gives another such endomorphism.
By Claim~4, both degrees
\[
 d_1=c+t\varepsilon_i-a,\qquad
 d_2=a+q\varepsilon_i-c
\]
are nonzero integer multiples of individual standard
basis vectors. Their sum is
$d_1+d_2=(q+t)\varepsilon_i\ne0.$
If $d_1=m\varepsilon_j$ with $j\ne i$ and $m\ne0$,
then $d_2=(q+t)\varepsilon_i-m\varepsilon_j$ would have
two nonzero coordinates, a contradiction. Thus both
$d_1$ and $d_2$ belong to $\mathbb Z\varepsilon_i$,
and consequently $c-a\in\mathbb Z\varepsilon_i$.
Fixing one nonsimple summand proves that every such
support is contained in the same set
$a+\mathbb Z\varepsilon_i$. If there is at most one
nonsimple summand, the assertion is immediate.

\smallskip
\noindent\textbf{Claim 6.}
Fix $i$, and let $U$ and $W$ be any two nonsimple summands
in the chosen decomposition of $f_iM$ into indecomposable
graded $B_i$-modules. Then at least one of the following holds:
\[
 \operatorname{supp}(U)\cap\operatorname{supp}(W)=\varnothing,
 \qquad
 \operatorname{supp}(U)\subseteq\operatorname{supp}(W),
 \qquad
 \operatorname{supp}(W)\subseteq\operatorname{supp}(U).
\]

\smallskip
\noindent\emph{Proof of Claim 6.}
Suppose that
$\operatorname{supp}(U)\cap\operatorname{supp}(W)\ne\varnothing$
and that neither support is contained in the other. Put
\[
 \begin{aligned}
 a&=\min\{v_i\mid v\in\operatorname{supp}(U)\},
 &
 b&=\max\{v_i\mid v\in\operatorname{supp}(U)\},\\
 c&=\min\{v_i\mid v\in\operatorname{supp}(W)\},
 &
 d&=\max\{v_i\mid v\in\operatorname{supp}(W)\}.
 \end{aligned}
\]
By Claim~2,
\[
 \{v_i\mid v\in\operatorname{supp}(U)\}
 =\{a,a+1,\ldots,b\},
 \qquad
 \{v_i\mid v\in\operatorname{supp}(W)\}
 =\{c,c+1,\ldots,d\}.
\]
Moreover, Claim~5 gives
\[
 v_j=w_j
 \qquad
 \bigl(
 v\in\operatorname{supp}(U),\
 w\in\operatorname{supp}(W),\
 j\ne i
 \bigr).
\]
Thus intersection and inclusion of the two supports are
equivalent to intersection and inclusion of their sets of
$i$th coordinates. Without loss of generality, we may assume
\begin{equation*}
 c<a\le d<b.
\end{equation*}

Choose $v\in\operatorname{supp}(U)$ with $v_i=a$, and put
$h=b-a+1$. By Claim~2, there is a homogeneous cyclic
generator $u\in U_v$ such that
\[
 U=B_i u,
 \qquad
 \operatorname{ann}_{B_i}(u)=J_i^h,
 \qquad
 \operatorname{soc}_{B_i}U
 =U_{v+(h-1)\varepsilon_i}.
\]
Since $c<a\le d$ and all coordinates other than the $i$th
agree on the two supports, Claims~2 and~5 imply that
$v\in\operatorname{supp}(W)$. Choose $0\ne w\in W_v$.

By Claim~1,
\[
 J_i^h w
 \subseteq
 \bigoplus_{p=h}^{\ell_i-1}W_{v+p\varepsilon_i}
 =0.
\]
Indeed, every degree occurring on the right has $i$th
coordinate at least
$a+h=b+1>d,$
whereas $d$ is the largest $i$th coordinate in
$\operatorname{supp}(W)$.
Since $\operatorname{ann}_{B_i}(u)=J_i^h$, the assignment
\[
 \theta:U\longrightarrow W,
 \qquad
 (xu)\theta=xw
 \quad(x\in B_i)
\]
is a well-defined $B_i$-module homomorphism. It is nonzero because
$(u)\theta=w\ne0$.
The map $\theta$ has degree zero.

Furthermore, the $i$th coordinate of
$v+(h-1)\varepsilon_i$ is $b>d$, so
\[
 (\operatorname{soc}_{B_i}U)\theta
 =
 (U_{v+(h-1)\varepsilon_i})\theta
 \subseteq
 W_{v+(h-1)\varepsilon_i}
 =0.
\]
Since $a>c$, Claim~2 also gives $W_v\subseteq J_iW$.
Consequently,
$(U)\theta=B_iw\subseteq J_iW.$

The summands $U$ and $W$ are distinct, since neither support contains the other. By the observation following Claim 3, extending $\theta$ by zero on the remaining B-summands gives a nonzero
degree-zero $\Lambda$-module endomorphism $\varphi$ of $M$ with $\varphi^2=0$. This contradicts the fact that $\End_{\mathrm gr \Lambda}(M)$ is a division ring. This proves Claim 6.

\smallskip
\noindent\emph{Completion of the proof.}
Fix $i$. We prove
\[
 \max_{v\in\operatorname{supp}(M)}v_i
 -
 \min_{v\in\operatorname{supp}(M)}v_i
 \le \ell_i-1.
\]

Suppose first that $J_iM=0$. By Claim~1, every homogeneous
component of $B_i$ of nonzero degree is contained in $J_i$
and therefore acts trivially on $M$. The degree-zero
component of $B_i$ preserves each subspace
\[
 \bigoplus_{v_i=n}M_v
 \qquad(n\in\mathbb Z).
\]
For $j\ne i$, every nonzero homogeneous component of $B_j$
has degree $q\varepsilon_j$ for some $q\ge0$, so its action
also preserves the $i$th coordinate. The same holds for $H$,
which is concentrated in degree zero by Claim~1.
Since $\Lambda=B+H$, each of the above subspaces is a
graded $\Lambda$-submodule. Hence
\[
 M=\bigoplus_{n\in\mathbb Z}
       \left(\bigoplus_{v_i=n}M_v\right)
\]
is a graded $\Lambda$-module decomposition.
The indecomposability of $M$ implies that precisely one
summand is nonzero. Thus
\[
 \max_{v\in\operatorname{supp}(M)}v_i
 -
 \min_{v\in\operatorname{supp}(M)}v_i
 =0,
\]
and the required inequality follows.

Suppose now that $J_iM\ne0$. Then the chosen decomposition
of $f_iM$ contains a nonsimple summand. Since this
decomposition is finite, we may choose a nonsimple summand
$U$ such that $\operatorname{supp}(U)$ is maximal under
inclusion among the supports of all nonsimple summands
of $f_iM$. Put
\[
 a=\min\{v_i\mid v\in\operatorname{supp}(U)\},
 \qquad
 b=\max\{v_i\mid v\in\operatorname{supp}(U)\}.
\]

Let $W$ be any nonsimple summand in the same decomposition.
If
$\operatorname{supp}(W)\cap\operatorname{supp}(U) \ne\varnothing,$
then Claim~6 and the maximality of
$\operatorname{supp}(U)$ imply
$\operatorname{supp}(W)\subseteq\operatorname{supp}(U).$
If instead the two supports are disjoint, Claims~2 and~5
imply that
\[
 \max_{v\in\operatorname{supp}(W)}v_i<a
 \qquad\text{or}\qquad
 \min_{v\in\operatorname{supp}(W)}v_i>b.
\]

Define the graded subspaces
\[
 M^-=\bigoplus_{v_i<a}M_v,
 \qquad
 M^0=\bigoplus_{a\le v_i\le b}M_v,
 \qquad
 M^+=\bigoplus_{v_i>b}M_v.
\]
The preceding argument shows that every nonsimple summand
of $f_iM$ is contained in one of these three subspaces.
By Claim~2, every simple summand has support consisting
of a single degree, so it is also contained in one of
these subspaces. Consequently, each of
\[
 f_iM\cap M^-,
 \qquad
 f_iM\cap M^0,
 \qquad
 f_iM\cap M^+
\]
is a direct sum of whole summands from the chosen graded
$B_i$-module decomposition of $f_iM$. In particular,
each is a $B_i$-submodule.

We verify that $M^-$, $M^0$, and $M^+$ are
$\Lambda$-submodules. Let $V$ be any one of these three
subspaces. Since $f_i$ has degree zero,
$f_iV=f_iM\cap V.$
For $x\in B_i$ and $m\in V$, it follows that
$xm=x(f_im)\in f_iM\cap V\subseteq V.$
Thus $V$ is stable under $B_i$.

For $j\ne i$, Claim~1 gives
\[
 B_j=\bigoplus_{q=0}^{\ell_j-1}(B_j)_{q\varepsilon_j}.
\]
Multiplication by an element of
$(B_j)_{q\varepsilon_j}$ preserves the $i$th coordinate,
so $B_jV\subseteq V$. Likewise, $HV\subseteq V$ because
$H\subseteq\Lambda_0$. Since $\Lambda=B+H$, the subspace
$V$ is a graded $\Lambda$-submodule.

We have therefore obtained a direct sum decomposition
$M=M^-\oplus M^0\oplus M^+$
into graded $\Lambda$-submodules. The chosen summand $U$
is contained in $M^0$, so $M^0\ne0$. The indecomposability
of $M$ now gives
$M=M^0.$
By Claim~2, the length $h$ of $U$ satisfies
$b-a=h-1\le\ell_i-1.$
Since every degree in $\operatorname{supp}(M)$ has
$i$th coordinate between $a$ and $b$, we conclude that
\[
 \max_{v\in\operatorname{supp}(M)}v_i
 -
 \min_{v\in\operatorname{supp}(M)}v_i
 \le b-a
 =h-1
 \le\ell_i-1.
\]
Since $i$ was arbitrary, the proposition follows.
\end{proof}

We now apply the preceding results to the algebra $A^0$
constructed in Lemma~\ref{c2:lem:replacement}. Its primitive corner
radicals can be assigned distinct coordinate directions, while the
off-diagonal components remain in degree zero.
Retain the notation of that lemma and assume that the residue
field $k$ is infinite. By Step~(1) of its proof and
\eqref{c2:eq:cleft-new-product},
\[
 A^0=B^0\oplus H,\qquad B^0=\prod_{i=1}^s A_i^0,
 \qquad
 A_i^0=
 \begin{cases}
 D_i[X_i;\sigma_i]/(X_i^{\ell_i}),&\ell_i>1,\\
 D_i,&\ell_i=1,
 \end{cases}
\]
where
$\ell_i=\operatorname{LL}(e_iAe_i) =\operatorname{LL}(A_i^0).$
Write $\mathfrak r_i^0=J(A_i^0)$, and set $X_i=0$ when
$\ell_i=1$. We identify $A_i^0$ with $e_i^0A^0e_i^0$
using \eqref{sec5u2}, and regard each $D_i$ as the
constant-coefficient subalgebra of $A_i^0$.
The multiplication in \eqref{c2:eq:cleft-new-product} makes
$D\oplus H$ a subalgebra of $A^0$, where $D=\prod_iD_i$.
The identities \eqref{sec5u}, \eqref{sec5u2}, and
\eqref{c2:eq:cleft-new-radical} give, respectively,
\begin{equation*}
 J(B^0)H=HJ(B^0)=0,\qquad
 e_i^0He_i^0=0,\qquad
 J(A^0)=J(B^0)\oplus H.
\end{equation*}
Let $\varepsilon_1,\ldots,\varepsilon_s$ be the standard
basis of $\mathbb Z^s$. Define
\begin{equation}\label{c2:eq:coordinate-decomposition}
 (A^0)_0=D\oplus H,\qquad
 (A^0)_{q\varepsilon_i}=D_iX_i^q
 \quad(1\le q<\ell_i),
\end{equation}
with all other homogeneous components equal to zero.
These subspaces define a $\mathbb Z^s$-grading on $A^0$:
multiplication within each $A_i^0$ respects the powers of
$X_i$, products between distinct factors of $B^0$ vanish,
and the positive-degree components annihilate $H$ on both
sides by \eqref{sec5u}.
In this grading, all the primitive idempotents $e_i^0$
have degree zero.

\begin{proposition}
\label{c2:prop:infinite-field}
Let $(R,\mathfrak m,k)$ be a commutative artinian local ring
with infinite residue field $k$, and let $A$ be a basic
Artin $R$-algebra satisfying $\X$.
Let $A^0$ be the finite-dimensional basic $k$-algebra
constructed in Lemma~\ref{c2:lem:replacement}.
Then both $A^0$ and $A$ are representation-finite.
\end{proposition}

\begin{proof}
If $A=0$, the construction in
Lemma~\ref{c2:lem:replacement} gives $A^0=0$, and there is
nothing to prove. Assume that $A\ne0$.
By Lemma~\ref{c2:lem:replacement}, $A^0$ is a
finite-dimensional basic $k$-algebra satisfying $\X$.
We equip it with the grading
\eqref{c2:eq:coordinate-decomposition}.
The idempotents $e_1^0,\ldots,e_s^0$ constructed in
Step~(1) of the proof of Lemma~\ref{c2:lem:replacement}
form a complete set of pairwise orthogonal primitive
idempotents and belong to $(A^0)_0$.

\smallskip
\noindent\emph{(1) We verify the hypotheses of
Propositions~\ref{c2:prop:finite-support-criterion}
and~\ref{c2:prop:bounded-decomposition}.}

Put $J^0=J(A^0)$.
By \eqref{c2:eq:cleft-new-radical} and \eqref{sec5u},
\[
 J^0=J(B^0)\oplus H,
 \qquad
 (J^0)^2=J(B^0)^2\oplus H^2.
\]
The second sum is direct because
$J(B^0)^2\subseteq B^0$, whereas
$H^2\subseteq H$ by \eqref{c2:eq:cleft-new-product}.
All these subspaces are homogeneous by
\eqref{c2:eq:coordinate-decomposition}.

For each $i$, the identities in \eqref{sec5u2} imply
\[
 e_i^0H^2e_i^0
 =\sum_{j=1}^s e_i^0He_j^0He_i^0
 =0.
\]
Indeed, the term with $j=i$ vanishes because
$e_i^0He_i^0=0$, and all the other terms vanish by the
same equation. Since $B^0=\prod_iA_i^0$, we obtain
\[
 e_i^0J^0e_i^0=\mathfrak r_i^0,
 \qquad
 e_i^0(J^0)^2e_i^0=(\mathfrak r_i^0)^2.
\]
Consequently, there is a graded isomorphism
\[
 e_i^0\bigl(J^0/(J^0)^2\bigr)e_i^0
 \cong\mathfrak r_i^0/(\mathfrak r_i^0)^2.
\]
The truncated skew polynomial description in Step~(1)
of the proof of Lemma~\ref{c2:lem:replacement}, together
with \eqref{c2:eq:coordinate-decomposition}, gives
\[
 \mathfrak r_i^0
 =\bigoplus_{q=1}^{\ell_i-1}D_iX_i^q,
 \qquad
 (\mathfrak r_i^0)^2
 =\bigoplus_{q=2}^{\ell_i-1}D_iX_i^q.
\]
Thus $e_i^0(J^0/(J^0)^2)e_i^0$ is concentrated in degree
$\varepsilon_i$. If $\ell_i=1$, this space is zero, so
the same assertion holds.

For $i\ne j$, one has $e_i^0B^0e_j^0=0$ because
$B^0=\prod_iA_i^0$. Thus the two radical decompositions above give
\[
 e_i^0\bigl(J^0/(J^0)^2\bigr)e_j^0
 \cong e_i^0He_j^0\big/e_i^0H^2e_j^0.
\]
This space is concentrated in degree zero because
$H\subseteq(A^0)_0$ by
\eqref{c2:eq:coordinate-decomposition}.
We have therefore verified all the hypotheses of
Proposition~\ref{c2:prop:bounded-decomposition} for $A^0$.
In particular,
\[
 e_i^0\bigl(J^0/(J^0)^2\bigr)_0e_i^0=0
 \qquad(1\le i\le s),
\]
so Proposition~\ref{c2:prop:finite-support-criterion}
also applies.

Finally, the corner identification in \eqref{sec5u2}
and Step~(1) of the proof of
Lemma~\ref{c2:lem:replacement} give
$\operatorname{LL}(e_i^0A^0e_i^0) =\operatorname{LL}(A_i^0)=\ell_i.$

\smallskip
\noindent\emph{(2) Application of the uniform support criterion.}

Set
\[
 I_0=\prod_{i=1}^s\{0,1,\ldots,\ell_i-1\}
 \subseteq\mathbb Z^s.
\]
Let $M$ be any nonzero indecomposable graded $A^0$-module,
and put
\[
 a_i=\min\{v_i\mid v\in\operatorname{supp}(M)\},
 \qquad a=(a_1,\ldots,a_s).
\]
By Proposition~\ref{c2:prop:bounded-decomposition} and
$\operatorname{LL}(e_i^0A^0e_i^0)=\ell_i$, proved in Step~(1),
\[
 0\le v_i-a_i\le\ell_i-1
 \qquad(v\in\operatorname{supp}(M),\ 1\le i\le s).
\]
Consequently,
$\operatorname{supp}(M[a]) =\operatorname{supp}(M)-a\subseteq I_0.$
The set $I_0$ is finite, nonempty, and independent of $M$.
The final assertion of
Proposition~\ref{c2:prop:finite-support-criterion} therefore
shows that $A^0$ is representation-finite.
By Lemma~\ref{c2:lem:replacement}, equivalently by the
indecomposable-class correspondence induced by
\eqref{c2:eq:cleft-iso-monoid}, $A$ is representation-finite
as well.
\end{proof}

\section{Proof of the main theorem}\label{c2:sec:completion}

We combine the reduction to a local central base ring with the two
field cases established in Section~\ref{sec:field-cases}.

\begin{proof}[Proof of Theorem~\ref{artin:main}]
Let $A$ be an Artin $R$-algebra.
We may assume that $A\ne0$.

First, we reduce to the case where $R$ is local.
Write
$R\cong R_1\times\cdots\times R_u,$
where each $R_v$ is a commutative artinian local ring.
Let $\varepsilon_1,\ldots,\varepsilon_u$ be the corresponding
central orthogonal idempotents, and put
$A_v=\varepsilon_v A.$
Omitting zero factors, we obtain
$A\cong A_1\times\cdots\times A_u,$
where $A_v$ is an Artin $R_v$-algebra.
Every indecomposable $A$-module is supported on exactly one
factor $A_v$. Inflation along the projection $A\to A_v$
preserves indecomposability, isomorphism classes, and
composition multiplicities. Hence, if $A$ satisfies $\X$, then $A_v$ satisfies $\X$
for every $v$.
Moreover, $A$ is representation-finite if and only if all
the factors $A_v$ are representation-finite.
It therefore suffices to prove the assertion when the
commutative artinian base ring is local.
Write this base ring as $(R,\mathfrak m,k)$, where
$k=R/\mathfrak m$.

Next, we may assume $A$ is basic.
By Lemma~\ref{c2:lem:replacement}, there exists a
finite-dimensional basic $k$-algebra $A^0$ satisfying
$\X$ such that
\[
 A\text{ is representation-finite}
 \quad\Longleftrightarrow\quad
 A^0\text{ is representation-finite}.
\]

If $k$ is infinite, Proposition~\ref{c2:prop:infinite-field}
gives the result.
If $k$ is finite, Proposition~\ref{c2:prop:finite-field}
shows that $A^0$ is representation-finite, and Lemma~\ref{c2:lem:replacement} gives the same conclusion for $A$.
\end{proof}

\bigskip
\noindent\textbf{Acknowledgements.} ChatGPT assisted the authors in exploring ideas and organizing some of the arguments during the preparation of this
paper.

\Needspace{12\baselineskip}

\bigskip
\noindent\begin{minipage}{\textwidth}
\normalsize
\setlength{\parindent}{1.5em}
\textsc{Liang Chen: School of Mathematical Sciences, Capital Normal
University, 100048 Beijing, China}\par
\textit{E-mail address:}
\href{mailto:2210501003@cnu.edu.cn}{\texttt{2210501003@cnu.edu.cn}}\par
\end{minipage}

\bigskip
\noindent\begin{minipage}{\textwidth}
\normalsize
\setlength{\parindent}{1.5em}
\textsc{Zhenxian Chen: School of Mathematical Sciences, Capital Normal
University, 100048 Beijing, China}\par
\textit{E-mail address:}
\href{mailto:czx18366459216@163.com}{\texttt{czx18366459216@163.com}}\par
\end{minipage}

\bigskip
\noindent\begin{minipage}{\textwidth}
\normalsize
\setlength{\parindent}{1.5em}
\textsc{Yongyi Wen: School of Mathematical Sciences, Capital Normal
University, 100048 Beijing, China}\par
\textit{E-mail address:}
\href{mailto:2240501002@cnu.edu.cn}{\texttt{2240501002@cnu.edu.cn}}\par
\end{minipage}

\bigskip
\noindent\begin{minipage}{\textwidth}
\normalsize
\setlength{\parindent}{1.5em}
\textsc{Weiheng Xia: School of Mathematical Sciences, Capital Normal
University, 100048 Beijing, China}\par
\textit{E-mail address:}
\href{mailto:xia.weiheng@qq.com}{\texttt{xia.weiheng@qq.com}}
\end{minipage}

\end{document}